\documentclass[letterpaper,11pt]{article}
\usepackage[margin=1in]{geometry}
\usepackage{amsmath,amsthm,amssymb,amsfonts}
\usepackage{tikz}
\usetikzlibrary{automata,arrows,positioning,calc}
\usepackage[capitalize]{cleveref}
\usepackage{enumitem}
\usepackage{color}
\usepackage{bbm}
\usepackage[subrefformat=parens]{subcaption}
\usepackage[normalem]{ulem}
\usepackage{xspace}
\usepackage{stmaryrd}
\usepackage{ifthen}
\usepackage{todonotes}

\newif\ifdraft
\drafttrue

\newcommand{\func}{\mathfrak{d}} 
\newcommand{\Den}{N} 

\newcommand{\Nat}{\mathbb{N}} 

\renewcommand\Pr{\mathop{\mathbf{Pr}}}
\DeclareMathOperator{\E}{\mathbf{E}}
\newcommand{\Var}{\mathop{\mathbf{Var}}}

\def\defeq{\mathrel{\mathop:}=}
\newcommand{\thit}{t_{\mathrm{hit}}}
\newcommand{\tmix}{t_{\mathrm{mix}}}
\newcommand{\tcov}{t_{\mathrm{cov}}}

\newcommand{\order}{\mathrm{o}}
\newcommand{\Order}{\mathrm{O}}

\def\inp<#1>{\langle #1 \rangle}
\def\Inp<#1>{\left\langle #1 \right\rangle}
\def\nor<#1>{\| #1 \|}
\def\Nor<#1>{\left\| #1 \right\|}
\def\bra<#1>{ \llbracket #1 \rrbracket }
\def\Bra<#1>{\left \llbracket #1 \right \rrbracket}

\newtheorem{theorem}{Theorem}[section]
\newtheorem{lemma}[theorem]{Lemma}
\newtheorem{corollary}[theorem]{Corollary}

\newtheorem{proposition}[theorem]{Proposition}

\crefname{equation}{}{}
\crefname{assumption}{Assumption}{Assumption}
\crefname{figure}{Figure}{Figure}

\title{How Many Vertices Does a Random Walk Miss in a
Network\\with Moderately Increasing the Number of Vertices?}
\author{
Shuji Kijima\thanks{Department of Informatics, 
   Graduate School of Information Science and Electrical Engineering, 
   Kyushu University, Fukuoka  819-0395, Japan; e-mail: {\ttfamily kijima@inf.kyushu-u.ac.jp}} \and
Nobutaka Shimizu\thanks{
Department of Mathematical Informatics, 
The University of Tokyo, 
Tokyo 113-8656, Japan; e-mail: {\ttfamily nobutaka\_shimizu@mist.i.u-tokyo.ac.jp}} \and 
Takeharu Shiraga\thanks{
Department of Information and System Engineering, 
Chuo University, Tokyo 112-8551, 
Japan; e-mail: {\ttfamily shiraga.076@g.chuo-u.ac.jp}}}
\date{\today}

\begin{document}
\maketitle 
\pagestyle{plain}
\baselineskip 17pt plus 1pt minus 1pt 


\begin{abstract}
 Real networks are often dynamic. 
 In response to it, 
  analyses of algorithms on {\em dynamic networks} 
  attract more and more attentions in network science and engineering. 
 Random walks on dynamic graphs also have been investigated actively in more than a decade, 
   where in most cases the edge set changes but the vertex set is static. 
 The vertex sets are also dynamic in many real networks. 
 Motivated by a new technology of the analysis of random walks on dynamic graphs, 
   this paper introduces a simple model of graphs with increasing the number of vertices, and 
   presents an analysis of random walks 
     associated with the cover time 
   on such graphs.  
 In particular, we reveal that 
    a random walk asymptotically covers the vertices all but a constant number  
     if the vertex set grows {\em moderately}. 

\ 

\end{abstract}

\section{Introduction}
 Networks appearing in the real world, 
    such as the Internet, 
     transportation networks, 
     sensor/wireless networks, 
     social networks and chemical dynamics, 
   change their shapes time by time. 
 Nevertheless, 
   what is known about the analyses of algorithms on dynamic networks is quite limited, 
   comparing with a wealth of knowledge on computations in static networks. 
 In response to it,  
   theoretical analyses of models and algorithms on dynamic networks 
     recently attract high attentions, 
     particularly in the context of  network science and engineering, 
    concerning 
  such as 
     connectivity, 
     exploration, %
      information spreading, gathering, agreement, 
      sampling, 
      population protocol, 
      random walks and other stochastic processes,  
        see e.g.,~\cite{MS18,Michail16,Augustine16,ClementiDC15,Sarma15,KO11,Cooper11}.

 Random walk on a graph is a fundamental stochastic process: 
  a walker on a vertex moves to a randomly picked neighbor at each discrete time step. 
 Random walk is a simple and powerful tool in the wide range of computer science, 
   such as randomized search, page rank and MCMC, 
   and so is it in networking science and engineering~\cite{Cooper11,Sarma15,AKL18,SZ19}. 
 The cover time of a random walk is 
   the time it takes for a walker to visit all vertices of the graph. 
 The cover time is one of the fundamental quantities of a random walk, 
   see e.g., \cite{Aleliunas79,Aldous83,Matthews88,Feige95up,Feige95low,DS84,AF02,LP17}, and 
    it is important with applications such as randomized search. 
   Analyses of {\em random walks on dynamic graphs}
    have been actively developed in the context, 
    where the  cover time is a central issue  
    \cite{Cooper11,CF03,AKL08,AKL18,DR14,YM16,LMS18,SZ19} 
  (see \cref{sec:related-work} for more detail).

 Those existing works, except for Cooper and Frieze~\cite{CF03}, 
   about random walks on dynamic networks 
    are concerned only with networks over a static vertex set. 
 However,  the real networks 
    change their vertex sets time by time. 
 Motivated by a new analysis technique, 
    this paper investigates random walks on graphs with increasing the number of vertices. 
 A dynamic vertex set causes some technical troubles:  
  it is questionable if 
    the ``cover time,'' that is a natural quantity for a static vertex set, 
    is also appropriate for a dynamic vertex set, and also 
  it is hopeless, as Cooper and Frieze~\cite{CF03} revealed, 
    to cover vertices beyond a constant ratio 
    when the number of vertices constantly increases. 
 In view of this, 
  we introduce a simple model of {\em growing graphs}, and 
  presents an analysis of the number of vertices remaining {\em unvisited} by a random walk 
   as a counterpart to the cover time of a random walk on a static vertex set.

\subsection{Model and quantities}\label{sec:model}
\paragraph{Example: collection of coupons with increasing the number of types.}
 To introduce our model, 
  let us start with a simple and intuitive example. 
 Suppose 
  you draw a coupon randomly from a finite number of types of coupons every day. 
  A single type of coupon exists on the first day, and
a new type of coupon is released at intervals of $n$ days
for the number $n$ of existing types of coupons,
i.e.,
you draw from two types of coupons for the second and the third days,
draw from three for the fourth to the sixth days, and
draw from $n$ for the $\binom{n}{2}+1$st to the $\binom{n+1}{2}$-th days.
 It might be difficult to {\em complete} all types of coupons 
   because new types are sequentially released. 
 Then, how many types of coupons do you expect to collect? 
 We will prove that 
   you can expect to {\em miss} at most two types of coupons. 
 On the other hand, interestingly, 
   the number of uncollected types of coupons diverges to infinity as the days go by 
   if the release intervals are $\order(n)$, e.g., $\lceil \sqrt{n} \rceil$ days 
    (see \cref{thm:complete_graph}). 

 Coupon collector's problem is often 
   connected to the cover time of a random walk on a complete graph. 
 Generalizing the above example, 
   we  investigate a random walk on a network with moderately increasing the number of vertices. 
 In the network model, 
   we introduce a parameter corresponding to the growth rate of the vertex set, 
    which will be represented by duration, in fact.  
 Then, we will be concerned with the number of unvisited vertices, 
    instead of the cover time. 
 
\paragraph{Random walk on a growing graph.}
A \emph{growing graph} is a sequence of graphs ${\cal G} = {\cal G}_0,{\cal G}_1,{\cal G}_2,\ldots$
  where each ${\cal G}_t= ({\cal V}_t,{\cal E}_t)$ is a connected simple undirected graph 
    such that ${\cal V}_t \subseteq {\cal V}_{t+1}$. 
  A {\em random walk} on a growing graph is 
    a stochastic process $Z=Z_0,Z_1,Z_2,\ldots$  ($Z_t \in {\cal V}_t$), where 
 the transition probability from $Z_t$ to $Z_{t+1}$ is 
   provided as a random walk on ${\cal G}_t$. 
 We remark that $Z_t \in {\cal V}_{t-1}$ holds for $t=1,2,\ldots$, in fact.  

 This paper is particularly concerned with 
     a simple model of growing graphs with moderate changes.  
 Roughly speaking, 
    a growing graph ${\cal G}$ in this paper 
     keeps being a graph\footnote{
  For instance, 
    $G^{(n)}$ is 
     a complete graph, 
     a path graph,  
     an expander graph, etc, of order $n$ respectively. 
  } $G^{(n)}$ unchanged for some duration of steps, then changes its shape to $G^{(n+1)}$ by adding a single vertex and connecting it to $G^{(n)}$. 
 Let $\func \colon \Nat\to\Nat$ be a function\footnote{E.g., $\func(n) = n$. },  
   denoting the {\em duration} of keeping the graph unchanged. 
 Then, 
  ${\cal G}$ is given as 
  ${\cal G}_t = G^{(n)}$ for $t$ satisfying $\sum_{i=1}^{n-1} \func(i)  \leq t < \sum_{i=1}^n \func(i)$ for $n=1,2,\ldots$, 
 where $G^{(n)}=(V^{(n)},E^{(n)})$ is a connected graph such that 
     $V^{(n)}=\{v_1,\ldots,v_n\}$ and 
     $E^{(n)} = E^{(n-1)} \cup \{\{ v_n, u \} : \mbox{ for some $u\in V^{(n-1)}$} \}$ for $v_n \in V^{(n)} \setminus V^{(n-1)}$.
 Notice that ${\cal G}_0$ is a graph of  a single vertex\footnote{
   This is just for convenience of descriptions, 
   but not essential in our later analyses. 
   See also \cref{sec:initial_round}. 
   }.
 In other words, 
  $\func(n)$ denotes the duration of $|{\cal V}_t| =n$, and hence  
  $\func(n) = \min \{t : |{\cal V}_t|=n+1 \} - \min\{t : |{\cal V}_t|=n \}$ holds.  
 For convenience, 
   let $T_n \defeq \sum_{i=1}^{n-1} \func(i) = \min \{t : |{\cal V}_t|=n \}$. 
 \cref{tbl:GtGn} shows the correspondence between ${\cal G}_t$ and $G^{(n)}$ in case of $\func(n)=n$. 

\begin{figure}[tbp]
\begin{center}
\begin{tabular}{|l|l|l|l|l|l|l|l|l|l|}
\hline
$Z_0$ & $Z_1$ & $Z_2$ & $Z_3$ & $Z_4$ & $Z_5$ & $Z_6$ & $Z_7$ & $Z_8$ & $\cdots$ \\ \hline
$\mathcal{G}_0$ & $\mathcal{G}_1$ & $\mathcal{G}_2$ & $\mathcal{G}_3$ & $\mathcal{G}_4$ & $\mathcal{G}_5$ & $\mathcal{G}_6$ & $\mathcal{G}_7$ & $\mathcal{G}_8$ & $\cdots$ \\ \hline
$G^{(1)}$        & $G^{(2)}$        & $G^{(2)}$        & $G^{(3)}$        & $G^{(3)}$        & $G^{(3)}$        & $G^{(4)}$        & $G^{(4)}$        & $G^{(4)}$        & $\cdots$ \\ \hline
\end{tabular}
\caption{Correspondence between $\mathcal{G}_t$ and $G^{(n)}$ when $\func(n)=n$. 
The transition from $Z_t$ to $Z_{t+1}$ is performed on ${\cal G}_t$, and hence $Z_{t+1} \in {\cal V}_t$ holds in fact.
In this example, $T_1=0$, $T_2=1$, $T_3=3$ and $T_4=6$.
\label{tbl:GtGn}}
\end{center}
\end{figure}

 This paper is also concerned with a particular model of random walks on growing graphs. 
 For simplicity, we assume that 
     a random walk on a growing graph ${\cal G}$ 
    is temporarily time-homogeneous, 
   meaning that   
   a random walk is formally represented by a common $n \times n$ transition matrix $P^{(n)}$ 
  such that $\Pr[Z_{t+1} = v \mid Z_{t} = u] = (P^{(n)})_{u,v}$ when ${\cal G}_t = G^{(n)}$. 
 We simply represent a random walk on a growing graph ({\em RWoGG}, for short) 
  by a triple $R=(\func,(G^{(n)})_{n=1}^{\infty}, (P^{(n)})_{n=1}^{\infty})$. 

 Then, we are concerned with the number of vertices unvisited by a RWoGG, 
   formally given by 
\begin{align*}
  {\cal U}_t 
   \defeq  \left| \left\{v \in {\cal V}_{t-1} : v \neq Z_s \mbox{ for any $s \in \{0,1,\ldots,t\}$} \right\}\right|
\end{align*}
 where 
 recall the fact that $Z_t \in {\cal V}_{t-1}$. 
 Particularly, 
  let $U(n)$ (or simply $U$ without confusion) 
   denote ${\cal U}_{T_{n+1}}$, 
 i.e., 
 $U(n) = n- \left|\bigcup_{t=0}^{T_{n+1}} \{Z_t \} \right|$, and 
  we will be concerned with it.  
 Remark that 
   ${\cal U}_t$ is monotone nonincreasing for $t \in (T_n,T_{n+1}]$, and 
   $U(n-1)+1 \geq {\cal U}_t \geq U(n)$ hold for the same time period.  

\paragraph{Terminology on time-homogeneous Markov chains.}
 We here briefly introduce other terminology 
    for random walks on static graphs, or time-homogeneous Markov chains, cf.~\cite{LP17}. 
 Suppose that $X_0,X_1,X_2,\ldots$  is a random walk on a static graph $G=(V,E)$ 
   characterized by a time-homogeneous transition matrix $P = (P_{u,v})\in [0,1]^{V\times V}$
    where $P_{u,v}=\Pr[X_{t+1}=v \mid X_t=u]$. 
 A transition matrix $P$ 
   is \emph{irreducible} 
     if $\forall u,v \in V$, $\exists t>0$, $(P^t)_{u,v} > 0$, and 
  is \emph{apperiodic} 
    if $\forall v \in V$, ${\rm GCD}\{t>0 : (P^t)_{v,v}>0\} = 1$. 
 An irreducible and apperiodic $P$ is said to be \emph{ergodic}. 
 A probabilistic distribution $\pi$ over $V$ is a \emph{stationary distribution} 
  if it satisfies $\pi P = \pi$.  
 It is well known that an ergodic $P$ has a unique stationary distribution~\cite{LP17}. 
 A random walk 
  is \emph{lazy} if $P_{v,v}\geq 1/2$ for all $v\in V$, 
  is \emph{reversible} if $\pi(u)P_{u,v}=\pi(v)P_{v,u}$ hold for all $u,v\in V$, and where $\pi\in[0,1]^V$ is the stationary distribution, and  
  is \emph{symmetric} if $P_{u,v}=P_{v,u}$ holds for all $u,v\in V$.
  A {\em simple} random walk (resp. {\em simple lazy} random walk) 
    on an undirected graph is given by $P_{u,v} = 1/d_u$ for $\{u,v\} \in E$
    (resp. $P_{u,v} = 1/(2d_u)$ for $\{u,v\} \in E$ and $P_{u,u} = 1/2$)
    where $d_u$ is the degree of $u$. 
 The {\em hitting time} $\thit$ (also denoted by $\thit(P)$) is given by 
    $\thit \defeq 
      \max_{u,v\in V}\E[\min\{t \geq 0 : X_0=u \mbox { and } X_t=v\}]$. 
 The {\em cover time} $\tcov$ (or $\tcov(P)$) is given by  
    $\tcov \defeq \max_{u\in V}\E[\min\{t \geq 0 : [X_0=u] \mbox { and } 
      [\forall v \in V,\, \exists s \leq t,\, X_s=v] \}]$. 
 The {\em mixing time}\footnote{
    Mixing time is usually parametrized by $\epsilon$, but we call $\tmix=\tmix(P)$ mixing time in this paper~\cite{LP17}.
  } $\tmix$ is given by 
  $\tmix \defeq \min\{t>0 : (1/2)\max_{u \in V}\sum_{v\in V}|P^t(u,v)-\pi(v)|\leq 1/4\}$.

\subsection{Our results}\label{sec:our-results}
 This paper investigates the behavior of $\E[U]$ regarding $\func$ 
   for a RWoGG $R=(\func,(G^{(i)})_{i=1}^{\infty},(P^{(i)})_{i=1}^{\infty})$, 
  where recall that $U$ is an abbreviation of $U(n) = {\cal U}_{T_{n+1}}$
   denoting the number of vertices unvisited by the random walk 
   at the moment just before a new vertex $v_{n+1}$ is attached (see \cref{sec:model} for precise). 
 Our results are summarized as follows.  

 \paragraph*{Complete graph {\rm (\cref{sec:complete_graph})}.}
 As an introductory example of our analyses, 
  we firstly concerned with a random walk on a growing complete graph, 
  which corresponds to the example 
    of collecting coupons with new releases in \cref{sec:model}.  
 Let 
   $R_{\rm c}=(\func,(G^{(i)})_{i=1}^{\infty},(P^{(i)})_{i=1}^{\infty})$ 
    be a random walk on a growing {\em complete} graph,  
   where 
    $G^{(i)}$ is a complete graph of order $i$, and 
    $(P^{(i)})_{u,v}=1/i$ for any $u \in V^{(i)}$ and $v \in V^{(i)}$ (including $u=v$). 
\begin{theorem} \label{thm:complete_graph}
 For $R_{\rm c}=(\func,(G^{(i)})_{i=1}^{\infty},(P^{(i)})_{i=1}^{\infty})$, 
   the following holds: 
\begin{enumerate}[label=(\arabic*)]
    \item\label{item:complete_Omega(i)} 
    If there is a constant $C>0$ such that $\func(i)\geq Ci$ for all $i\in [n]$,
    then $\E[U]=\Order(1)$.
    \item\label{lab:complete_omega(i)}
    If $\func(i)/i\to\infty$ as $i\to\infty$,
    then $\E[U]\to 0$ as $n\to\infty$.
    \item\label{item:complete_o(i)} 
    If $\func$ is unbounded (i.e., $\func(i)\to\infty$ as $i\to\infty$) and 
      satisfies for all $i\in\Nat$ that 
       $\frac{\func(i)}{i}\geq \frac{\func(i+1)}{i+1}$ and 
       $\func(i)\leq \func(i+1)$,
    then $\E[U]=(1-\order(1))\frac{n}{\func(n)+1}$. 
    \item\label{item:complete_const}
    If $\func$ is constant (i.e., $\exists c\in\Nat$, $\forall i\in\Nat$, $\func(i)=c$), 
    then $\E[U]=(1-\Order(n^{-1}))\frac{n}{c+1}$.
\end{enumerate}
\end{theorem}
 Notice that 
  \ref{item:complete_Omega(i)} implies that 
     the number of missing types of coupons is at most a constant in expectation, 
     i.e., $\E[{\cal U}_t]=\Order(1)$ at any time $t$, if $\func(i)=\Omega(i)$, 
  while \ref{lab:complete_omega(i)} claims 
    a stronger upper bound     with a stronger assumption of $\func(i)=\omega(i)$ 
   that the expected number of missing types is asymptotic to 0 every time just before a new release 
       (recall the relation between $U$ and ${\cal U}_t$). 
 \ref{item:complete_o(i)} claims in case of $\func(i)=\order(i)$ and $\omega(1)$ that 
   $\E[U] \approx \frac{n}{\func(n)}$ up to the leading coefficient; 
  for instance,  
  $\E[U] \leq n^{\gamma}/C$ holds if $\func(i) \geq Ci^{1-\gamma}$ as well as  
  $\E[U] \geq n^{\gamma}/C$ holds if $\func(i) \leq Ci^{1-\gamma}$, 
 where $C>0$ and $\gamma\in [0,1]$ are arbitrary constants common in both equations (See also \cref{cor:simpleKn}).
 \ref{item:complete_const} is the counterpart of \ref{item:complete_o(i)} for constant $\func$. 
For example, if a new vertex appears every step ($\func(i)=1$), a random walk on a growing complete graph misses a half of the number vertices.

\paragraph*{Upper bound analysis {\rm (\cref{sec:general_graphs})}. }
 Next, 
  we focus on upper bounds of $\E[U]$ with respect to $\func$ 
   for RWoGG $(\func,(G^{(i)})_{i=1}^{\infty},(P^{(i)})_{i=1}^{\infty})$, 
   in general. 
 For convenience, 
  let
   $\thit(i)$, $\tcov(i)$ and $\tmix(i)$ 
  respectively denote the hitting, cover and mixing times of $P^{(i)}$, 
 in the rest of the paper. 

To begin with, we remark that it is easy
to prove that $\E[U]=\Order(1)$ if $\func(i)=\Omega(\thit(i) \log i)$
for any RWoGG
using the known fact
that the number of unvisited vertices exponentially decays every unit
time of ${\rm e}\thit$
(see e.g.~Sections 2.4.3 and 2.6 of \cite{AF02}; see also
\cref{lem:static_unvisited} in \cref{app:tools}).
 Thus, our interest is in the case that  $\func(i)=\order(\thit(i) \log i)$. 
 We establish the following upper bound of $\E[U]$, 
   claiming that  $\E[U]=\Order(1)$ if $\func(i) = C\thit(i)$ for $C>1$, in fact. 
 We remark that the following theorem 
    is an extension of 
    \cref{thm:complete_graph}~\ref{item:complete_Omega(i)} and \ref{lab:complete_omega(i)} 
    for ``a specific random walk on growing complete graphs''  
    to general random walks and graphs. 
 \begin{theorem}
\label{thm:simplified_main}
 Let $(\func,(G^{(i)})_{i=1}^{\infty},(P^{(i)})_{i=1}^{\infty})$ be an arbitrary RWoGG. 
\begin{enumerate}[label=(\arabic*)]
  \item \label{lab:general} 
  If there is a constant $C>1$ such that $\func(i)\geq C\thit(i)$ for all $i\in [n]$, 
  then $\E[U]=\Order(1)$.
  \item \label{lab:general_omega}
  If $\func(i)/\thit(i)\to \infty$ as $i\to \infty$, 
  then $\E[U]\to 0$ as $n\to \infty$.
\end{enumerate}
\end{theorem}

In \cref{thm:simplified_main}, we obtain a {\em general} upper bound of $\E[U]$ in the case of $\func(i)\geq (1+\epsilon)\thit(i)$, where $\epsilon>0$ is a constant.
In contrast, 
the case of $\func(i) \leq (1+\order(1))\thit(i)$ seems not easy: 
    it contains an issue of ``short random walks,'' 
   that is a challenging topic in the literature of the cover time of multiple random walks, and so on,  
    see e.g., \cite{KMS19}. 
 Henceforth, we focus on lazy and reversible random walks, 
    of which the transition matrices $P^{(i)}$ are known to be (essentially\footnote{
    The transition matrix  $P$ of a lazy and reversible random walk is not symmetric in general, 
    but there always exists a diagonal matrix $D$ such that$D^{-1}PD$ is symmetric~see e.g., \cite{LP17}. 
    }) positive semidefinite. 
  For ``rapidly'' mixing random walks such that $\tmix \ll \thit$, 
   we obtain the following upper bound. 
\begin{theorem} \label{thm:tmix<<thit}
 Let $(\func,(G^{(i)})_{i=1}^{\infty},(P^{(i)})_{i=1}^{\infty})$ be a RWoGG 
   such that $P^{(i)}$ is lazy and reversible. 
 Let $C>0$ and $\gamma\in [0,1]$ be arbitrary constants.
 If $\thit(i)/\tmix(i)\geq i^{\gamma}/C$ and 
  $\func(i)\geq \frac{3C\thit(i)}{i^\gamma}$
 for all $1<i\leq n$, 
  then $\E[U]\leq \frac{8n^\gamma}{C}+32$.
\end{theorem}
 Notice that \cref{thm:tmix<<thit} for $\gamma=0$ claims 
   that $\E[U] = \Order(1)$ if $\func(i)=\Theta(\thit(i))$ on the appropriate condition. 
 A natural question remains unsettled 
   whether  $\E[U] = \Order(1)$ requires $\func(i)=\Omega(\thit(i))$ 
     for any RWoGG $(\func,(G^{(i)})_{i=1}^{\infty},(P^{(i)})_{i=1}^{\infty})$. 
 As a consequence of  \cref{thm:tmix<<thit}, for example, 
  we obtain a bound for degree restricted expander graphs, 
    for which $\thit(i)=\Order(i)$ 
    and $\tmix(i)=\Order(\log i)$ hold, 
  that $\E[U]=\Order(n^{\gamma})$ if $\func(i)=\Omega(i^{1-\gamma})$ for $\gamma\in [0,1)$; 
   see \cref{cor:expander}, for detail. 
 We also remark that 
   the upper bound by \cref{thm:tmix<<thit} is tight for growing complete graphs, 
    for which $\thit(i)=\Theta(i)$ and $\tmix(i)=\Theta(1)$ hold; 
   see \cref{thm:complete_graph}~\ref{item:complete_o(i)}  (\cref{cor:simpleKn}\ref{lab:simpleKnlower}) for the lower bound. 

 Though the condition of $\tmix \ll \thit$ covers 
   many interesting examples of rapidly mixing random walks,  
   it misses many examples, such as random walks on paths and lollipop graphs, 
    interested in the context of hitting and cover times. 
 Then, we provide for those examples the following \cref{thm:symmetry,thm:simplelazyRW}.

\begin{theorem}
\label{thm:simplelazyRW}
 Let $(\func,(G^{(i)})_{i=1}^{\infty},(P^{(i)})_{i=1}^{\infty})$ be a RWoGG 
   such that $P^{(i)}$ is lazy and simple, and  
   that 
    for all $i$ ($2<i\leq n$), $\frac{|E^{(i)}|}{|E^{(i-1)}|}\leq 1+\frac{L}{i}$ hold for some positive constant $L$.
Let $C>0$ and $\gamma\in [0,1]$ be arbitrary constants.
If $\func(i)\geq \left(\frac{C}{i^\gamma}+\frac{L+1}{2i}\right)\thit(i)$ holds for any $1<i\leq n$, then $\E[U]\leq \sqrt{L+1}\frac{n^{\gamma}}{C}$.
\end{theorem}
 We will later give a tight example for \cref{thm:simplelazyRW}; 
 \cref{thm:path_lowerbound} gives a lower bound of $\E[U]$ for a growing {\em path} 
  (see also~\cref{cor:path}). 
 We will also demonstrate another example of 
   application of \cref{thm:simplelazyRW} to a growing {\em lollipop graph} (see \cref{thm:lollipop}), 
  where the static lollipop graph is well-known as 
   a tight example for the bounds $\thit=\Order(n^3)$ and $\tcov = \Order(n^3)$ 
   for a simple random walk for any graph. 
 
\begin{theorem}\label{thm:symmetry}
 Let $(\func,(G^{(i)})_{i=1}^{\infty},(P^{(i)})_{i=1}^{\infty})$ be a RWoGG 
  such that $P^{(i)}$ is lazy and symmetric. 
 Let $C>0$ and $\gamma\in [0,1]$ be arbitrary constants.
 If $\func(i)\geq \left(\frac{C}{i^\gamma}+\frac{2}{i}\right)\thit(i)$ for all $1<i\leq n$, 
  then $\E[U]\leq \frac{\sqrt{3}n^\gamma}{C}$.
\end{theorem}
 A typical application of \cref{thm:symmetry} is  
   a lazy Metropolis walk with the uniform stationary distribution (see \cref{thm:Metro} for details), 
    which often appears in the context of Markov chain Monte Carlo. 
 Nonaka et al.~\cite{Nonaka10} proved 
   that the Metropolis achieves $\thit(i)=\Order(i^2)$ for any connected graph. 
 The upper bound by \cref{thm:symmetry} is also tight 
   for a Metropolis walk on a growing path (\cref{thm:path_lowerbound,cor:path}).

\paragraph*{A lower bound for a growing path {\rm (\cref{sec:path})}.}
In contrast to upper bounds, 
   an analysis of a lower bound requires more technically complicated arguments. 
 We establish a lower bound of $\E[U]$ for a random walk on a growing path graph, 
  which implies that the upper bound by \cref{thm:symmetry} is tight in the case. 
Let $R_{\rm p}=(\func,(G^{(i)})_{i=1}^\infty,(P^{(i)})_{i=1}^\infty)$ 
   be a random walk on a growing {\em path} graph, 
    where $G^{(i)}=(V^{(i)},E^{(i)})$ is given by $V^{(i)}=\{v_1,\ldots,v_i\}$, 
       and $E^{(i)}=\{\{v_1,v_2\}$, $\ldots$, $ \{v_{i-1},v_i\}\}$, and 
 $P^{(i)}$ is given by 
\begin{align} \label{eqn:transition_matrix_path}
    (P^{(i)})_{u,v} = \begin{cases}
    p & \text{if $u=v=v_1$ or $u=v=v_i$},\\
    1-p & \text{if $(u,v)\in\{(v_1,v_2),(v_{i},v_{i-1})\}$},\\
    q & \text{if $\{u,v\} = \{v_j,v_{j+1}\}$ for $j=2,3,\ldots,i-1$},\\
    1-2q & \text{if $u=v=v_j$  for $j=2,3,\ldots,i-1$},\\
    0 & \text{otherwise}
    \end{cases}
\end{align}
for two parameters $p,q\in[0,1]$ satisfying $p \geq q$ and $q\leq 1/2$ (see \cref{fig:chain}).
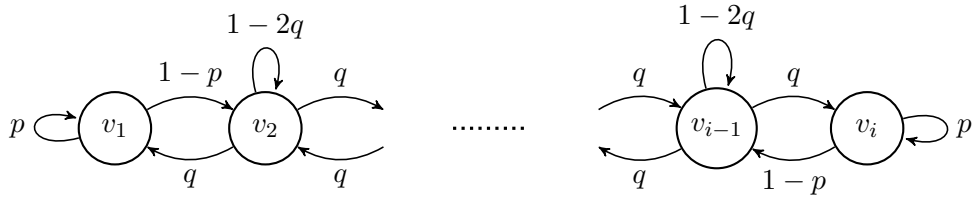
\begin{figure}[tbp]
\begin{center}
\begin{tikzpicture}[->, >=stealth', auto, semithick, node distance=2cm]
\tikzstyle{every state}=[fill=white,draw=black,thick,text=black,scale=1]
\node[state]    (v1)               {$v_1$};
\node[state]    (v2)[right of=v1]   {$v_2$};
\node[state,draw=white]    (v3)[right of=v2]   {};
\node[state,draw=white]    (v4)[right of=v3]   {};
\node[state]    (v5)[right of=v4]   {$v_{i-1}$};
\node[state]    (v6)[right of=v5]   {$v_i$};
\draw[-,dotted,very thick] (v3)--(v4);
\path
(v1) edge[loop left]     node{$p$}           (v1)
     edge[bend left]     node{$1-p$}         (v2)    
(v2) edge[bend left]     node{$q$}           (v1)
(v2) edge[loop above]    node{$1-2q$}        (v2)
(v2) edge[bend left]     node{$q$}           (v3)
(v3) edge[bend left]     node{$q$}           (v2)
(v4) edge[bend left]     node{$q$}           (v5)
(v5) edge[loop above]    node{$1-2q$}        (v5)
(v5) edge[bend left]     node{$q$}           (v4)
(v5) edge[bend left]     node{$q$}           (v6)
(v6) edge[bend left]     node{$1-p$}         (v5)
(v6) edge[loop right]    node{$p$}           (v6);
\end{tikzpicture}
\end{center}
\caption{The transition diagram of \cref{eqn:transition_matrix_path}.\label{fig:chain}}
\end{figure}
For example, if $(p,q)=(\frac{1}{2},\frac{1}{4})$, the corresponding walk is the lazy simple random walk.
If $(p,q)=(\frac{3}{4},\frac{1}{4})$ the corresponding one is the lazy Metropolis random walk (see \cref{def:Metro} for the definition of Metropolis random walk).

Suppose, for instance, that $\func(i)=Ci$ for a sufficiently large constant $C>0$.
Then, the walker walks
$\sum_{i=1}^n \func(i)\approx C^2n^2/2$ steps in total,
which is larger than the cover time of a lazy simple random walk on the path of length $n$.
Thus, one may expect that $\E[U]=\Order(1)$. 
However, this is not the case. 

\begin{theorem} \label{thm:path_lowerbound}
 If $\func(i)\leq Ci^{2-\gamma}$ in $R_{\rm p}$ for some constants $C>0$ and $\gamma\in[0,1]$ 
    then $\E[U]=\Omega(n^{\gamma}/C)$.
\end{theorem}

 \cref{thm:simplelazyRW,thm:symmetry,thm:path_lowerbound} imply the following tight bounds of $\E[U]$ on a growing path.
\begin{corollary} \label{cor:path}
 For $R_{\rm p}=(\func,(G^{(i)})_{i=1}^{\infty},(P^{(i)})_{i=1}^\infty$, where $P^{(i)}$ is the transition matrix of either the lazy simple random walk or the lazy Metropolis random walk.
Then
\begin{enumerate}[label=(\arabic*)]
    \item\label{item:path_upper} 
      If $\func(i)\geq Ci^{2-\gamma}$ for some constants $C>0$ and $\gamma\in[0,1]$ 
    then $\E[U]=\Order(n^{\gamma}/C)$.
    \item\label{item:path_lower} 
      If $\func(i)\leq Ci^{2-\gamma}$ for some constants $C>0$ and $\gamma\in[0,1]$ 
    then $\E[U]=\Omega(n^{\gamma}/C)$.
\end{enumerate}
\end{corollary}


\subsection{Related works}\label{sec:related-work}
 The cover time is a fundamental topic of analyses of random walks. 
 Here, we review some representative results 
   about the cover times of random walks on static graphs, and on dynamic graphs.  
   
\paragraph{Cover times of random walks on static graphs.}
 It is known that the cover time of a simple random walk satisfies 
  $\tcov \leq 2m(n-1)$ 
  for any undirected graph, 
   see Aleliunas et al.~\cite{Aleliunas79} and Aldous~\cite{Aldous83}. 
  Mathews~\cite{Matthews88} 
   devised a technique of upper and lower bounding $\tcov$ by $\thit$, 
    of which a celebrated implication is $\tcov \leq \thit \log n$.
 The lolipop graph is famous 
   for  $\thit = \Omega(n^3)$, and hence $\tcov = \Omega(n^3)$. 
  Fiege gave a tight upper bound of the cover times of simple random walks on any graphs 
    such that $\tcov  \leq \frac{4}{27}n^3 + \Order(n^{5/2})$ in~\cite{Feige95up}, 
  while he in \cite{Feige95low} 
    gave a tight lower bound of the cover time of simple random walks on any graphs 
   such that $\tcov  \geq n \ln n + \order(n \ln n)$,  
     using a Mathews' argument~\cite{Matthews88}. 
The connection between the hitting time and electric circuits is well known (see e.g.,~\cite{DS84,AF02,LP17}). 

 Motivated by a faster covering by a random walk,  
 Ikeda et al.~\cite{Ikeda03} (see also \cite{Ikeda09}) proposed $\beta$-random walk, 
   which makes transitions only using local information, and proved that 
   the cover time of  a $\beta$-random walk is upper bounded by $\Order(n^2 \log n)$ for any graph. 
 Nonaka et al.~\cite{Nonaka10} proved the same bound holds for a Metropolis walk, 
   which is simpler and more popular than $\beta$-random walk.  
 Recently, 
   David and Feige~\cite{DF17} (see also \cite{DF18}) proved that 
   a biased random walk achieves $\Order(n^2)$ cover time for any graph, and 
   affirmatively settled the question posed by Ikeda et al.~\cite{Ikeda03}. 


\paragraph{Cover time of random walks on dynamic graphs.}
 An early work~\cite{CF03} by Cooper and Frieze 
   investigated random walks on ``web-graphs,'' 
    where the number of vertices increases every constant steps, 
     i.e., corresponding to constant $\func$ in our model, and 
     where $G^{(n)}$ is 
     a preferential attachment graph. 
 Then,  
   they were concerned with the expected 
     proportion of vertices visited by a random walk, and 
  they revealed that it converges to some constant 
   accordingly $\E[U]/n$ converges to some constant in our context, 
  asymptotic to $n$. 

 There are several results 
    about the cover  times of random walks on dynamic graphs, 
       sometimes called ``evolving graphs,'' 
     with static vertex sets. 
 Avin et al.~\cite{AKL08} (see also \cite{AKL18}) 
   investigated the hitting times, mixing times and cover times of 
    random walks on evolving graphs with static vertex sets. 
 They gave a prescribed sequence of graphs on which 
   the hitting time of a simple random walk gets $2^{\Omega(n)}$, 
     and hence the cover time is as well.  
 On the other hand, 
    they proved that 
      the cover time of a max-degree random walk is $\Order(d_{\max}n^3(\log n)^2)$ 
      where $d_{\max}$ is the maximum degree of the evolving graph. 
  Denysyuk and Rodrigues~\cite{DR14}
    were concerned with {\em $\rho$-recurrent family} of evolving graphs, 
      where preferable graphs are assumed  to appear frequently in the graph sequence. 
 Then, for max-degree random walks on $\rho$-recurrent families, 
  they gave upper and lower bounds of the cover time in terms of the hitting time, 
    as well as gave an upper bound of the mixing time. 
  Lamprou et al.~\cite{LMS18} 
   were concerned with two random walks of 
    ``random walk with a delay'' (RWD),  
     where at  each step, the walker chooses an edge of underlying graph and moves when it appears, and 
    ``random walk on what is available'' (RWA),  
     where the walker chooses an edge of current graph and moves immediately. 
  Then, 
    they investigated the cover times of RWD and RWA 
     for edge-uniform stochastically evolving graphs. 
  Sauerwald and Zanetti~\cite{SZ19} 
    extended the argument by Avin et al.~\cite{AKL18} 
      in the case that a sequence of graphs have the same stationary distribution, and 
    presented an upper bound $\Order(n^2)$ 
     of the cover time on $d$-regular dynamic graphs.

\paragraph{Other related works.}
 Saloff-Coste and Z\'{u}\~{n}iga 
   investigated time-inhomogeneous Markov chains, and 
   provided  some Nash and log-Sobolev inequalities \cite{Saloff-Coste09,Saloff-Coste11}. 
  Recently, 
   Cai et al.~\cite{CSZ20} investigated the relation 
     between the density of edge-Markovian dynamic graphs and mixing times. 
 They showed for fast-changing dynamic graphs 
    that $\tmix = \infty$ in sparse case 
    while $\tmix = \Order(\log n)$ in dense case.  
 They also showed for slowly-changing dynamic graphs 
    that $\tmix = \Omega(n)$ in sparse case 
    while $\tmix = \Order(\log n)$ in dense case.  
 Random walk on dynamic graph is also interested in data mining. 
  Yu and McCann~\cite{YM16} presented an analysis 
    on ``random walk with restart,''
     which is used 
          as a measure of proximity between vertices of a graph in the context,  
   over dynamic graphs. 

 There are many works on other stochastic processes on dynamic graphs, 
  such as exploration, information spreading, rumor spreading, gossiping and voter model, 
    see e.g.,  \cite{Augustine16,ClementiDC15,BGKM16}. 
 Theoretical analyses of algorithms on dynamic graphs 
     attract high attentions in the context of distributed computing, and 
  there are many works concerning the topics, 
    such as 
     connectivity, exploration, gathering, agreement,  flooding and 
     population protocol, on dynamic networks, 
      see e.g., ~\cite{MS18,Michail16,KO11}.  
 

\section{Complete Graph} \label{sec:complete_graph}
This section proves \cref{thm:complete_graph}. 
Throughout this paper, we consider a random walk of length $T_{n+1}$.
For convenience, we divide the $T_{n+1}$ step random walk into $n$ random walks each of length $\func(i)$ (for $i=1,\ldots,n$).
We call each period \emph{round}.
For a round $i\in[n]$,
let $(X^{(i)}_s)_{s=0}^\infty$ denote
   a random walk in the $i$-th round (specifically, it is a random walk according to $P^{(i)}$) with the initial state $X^{(i)}_0=Z_{T_i}= X^{(i-1)}_{\func(i-1)}$.
Note that $(X^{(i)}_s)_{s=0}^\infty$ is a random walk on $G^{(i)}$.
 \Cref{tbl:ZtoX} illustrates the correspondence between $Z_t$ and $X^{(i)}_s$ in the case of $\func(i)=i$.
 
%
\begin{figure}[tbp]
\begin{center}
\begin{tabular}{|l|l|l|l|l|l|l|l|l|l|l|}
\hline
          & $Z_0$       & $Z_1$       & $Z_2$       & $Z_3$       & $Z_4$       & $Z_5$       & $Z_6$       & $Z_7$   & $Z_8$   & $\cdots$ \\ \hline
$G^{(1)}$ & $X^{(1)}_0$ & $X^{(1)}_1$ & $\cdots$    &             &             &             &             &             &      &    \\ \hline
$G^{(2)}$ &             & $X^{(2)}_0$ & $X^{(2)}_1$ & $X^{(2)}_2$ & $\cdots$    &             &             &             &      &    \\ \hline
$G^{(3)}$ &             &             &             & $X^{(3)}_0$ & $X^{(3)}_1$ & $X^{(3)}_2$ & $X^{(3)}_3$ & $\cdots$    &      &    \\ \hline
$G^{(4)}$ &             &             &             &             &             &             & $X^{(4)}_0$ & $X^{(4)}_1$ &$X^{(4)}_2$ &$\cdots$  \\ \hline
\end{tabular}
\caption{Correspondence between $Z_t$ and $X^{(i)}_s$ when $\func(i)=i$.
For each $i\in\Nat$, $(X^{(i)}_s)_{s=0,1,\ldots}$ is a random walk on $G^{(i)}$.
Note that $X^{(i)}_0=X^{(i-1)}_{\func(i-1)}=Z_{T_i}$ holds for $i=2,3,\ldots$.
In this example, $U(3)=3-\left|\bigcup_{t=0}^{T_{3+1}} \{Z_t \} \right|=3-\left|\bigcup_{i=1}^{3}\bigcup_{s=0}^{i}\{X_s^{(i)} \} \right|$.
\label{tbl:ZtoX}}
\end{center}
\end{figure}

 For $v\in V^{(n)}$ let $\mathcal{E}(v)$ denote
the event that $v\not\in \bigcup_{i=1}^n \bigcup_{s=0}^{\func(i)} \{X_s^{(i)}\}$ 
  ($=\bigcup_{t=0}^{T_{n+1}} \{Z_t\}$).  
In other words,
$\mathcal{E}(v)$ means that the random walk $Z_0,Z_1,\ldots, Z_{T_{n+1}}$  
does not visit the vertex $v$.
For the vertex $v_k$ attached to ${\cal G}$ at time $T_k$, 
we see that 
$\Pr[\mathcal{E}(v_k)]=\prod_{i=k}^n\left(1-\frac{1}{i}\right)^{\func(i)}$ holds, and thus
\begin{align*}
    \E[U] &= \sum_{k=1}^n \Pr[\mathcal{E}(v_k)]
    = \sum_{k=1}^n \prod_{i=k}^n\left(1-\frac{1}{i}\right)^{\func(i)}
\end{align*}
 holds. \Cref{thm:complete_graph} follows the next lemma.
\begin{lemma} \label{lem:Snbound}
For a function $f:\Nat\to\Nat$, let
$S(n)\defeq \sum_{k=1}^n \prod_{i=k}^n \left(1-\frac{1}{i}\right)^{f(i)}$.
\begin{enumerate}[label=(\roman*)]
    \item\label{item:upper} If $f(i)\geq Ci$ for some constant $C$,
    then $S(n)=\Order(1)$.
    \item\label{item:tightlowerbound} If $f$ satisfies $f(i)\leq f(i+1)$ for all $i\in\Nat$,
    then $S(n)\geq \frac{n}{f(n)+1}\left(1-\frac{1}{n}\right)^{f(n)}$.
    \item\label{item:tightupperbound} If $f$ satisfies $\frac{f(i)}{i}\geq \frac{f(i+1)}{i+1}$,
    then for all $n\in\Nat$,
    $S(n)\leq \frac{n}{f(n)}$.
    \item\label{item:tightupperbound2}
    If there is a constant $c\in\Nat$ such that $f(i)=c$ for all $i\in\Nat$,
    then for all $n\in\Nat$,
    $S(n)\leq \frac{n}{c+1}$.
\end{enumerate}
\end{lemma}
\begin{proof}[Proof of \ref{item:upper}]
Since $1+x\leq \mathrm{e}^x$, we have
\begin{align*}
    S(n) &\leq \sum_{k=1}^n \exp\left(-\sum_{i=k}^n\frac{f(i)}{i}\right)
    \leq \sum_{k=1}^n \exp\left(-(n-k+1)C\right) = \Order(1).
\end{align*}
\end{proof}
\begin{proof}[Proof of \ref{item:tightlowerbound}]
Observe that $S(1)=0$ and for all $n\geq 1$,
\begin{align}
S(n+1)&= \sum_{k=1}^{n+1} \prod_{i=k}^{n+1} \left(1-\frac{1}{i}\right)^{f(i)} 
=\left(1-\frac{1}{n+1}\right)^{f(n+1)}\left(S(n)+1\right).
\label{eqn:rec}
\end{align}
We prove \ref{item:tightlowerbound} by induction on $n$.
In the base case, $S(1)=0$ and we are done.
If $S(n)\geq \frac{n}{f(n)+1}\left(1-\frac{1}{n}\right)^{f(n)}$, then
\begin{align}
S(n)+1&\geq \frac{n}{f(n)+1}\left(1-\frac{1}{n}\right)^{f(n)}+1
\geq \frac{n}{f(n)+1}\left(1-\frac{f(n)}{n}\right)+1 \nonumber \\
&=\frac{n-f(n)}{f(n)+1}+1
=\frac{n+1}{f(n)+1}\geq \frac{n+1}{f(n+1)+1}. \label{eqn:lower1}
\end{align}
Here, we used $(1+x)^r\geq 1+rx$ in the second inequality and $f(n)\leq f(n+1)$ in the last inequality.
Combining \cref{eqn:rec,eqn:lower1}, $S(n+1)\geq \left(1-\frac{1}{n+1}\right)^{f(n+1)}\frac{n+1}{f(n+1)+1}$ and we are done.
\end{proof}
\begin{proof}[Proof of \ref{item:tightupperbound}]
The proof is obtained by induction on $n\geq 1$. 
When $n =1$, $S(1)=0\leq 1/f(1)$.
Assume $S(n)\leq n/f(n)$. 
Then,  
\begin{align*}
S(n+1)&=\left(1-\frac{1}{n+1}\right)^{f(n+1)}\left(S(n)+1\right)
\leq \frac{\frac{n}{f(n)}+1}{1+\frac{f(n+1)}{n+1}}
\leq \frac{\frac{n+1}{f(n+1)}+1}{1+\frac{f(n+1)}{n+1}}=\frac{n+1}{f(n+1)}.
\end{align*}
Note that $(1-x)^y\leq 1/(1+xy)$ for all $x\in [0,1]$ and $y\geq 0$.
The second inequality follows from $\frac{f(n+1)}{n+1}\leq \frac{f(n)}{n}$.
\end{proof}
\begin{proof}[Proof of \ref{item:tightupperbound2}]
The proof is obtained by induction on $n$.
First $S(1)=0\leq 1/(f(1)+1)$.
Assume $S(n)\leq n/(f(n)+1)$.
Then, from \cref{eqn:rec} and the induction assumption, we have
\begin{align*}
S(n+1)
&\leq \frac{\frac{n}{f(n)+1}+1}{1+\frac{f(n+1)}{n+1}}
=\frac{\frac{n}{f(n)+1}+1}{1+\frac{f(n)}{n+1}}
=\frac{\frac{n+1}{f(n)+1}\left(\frac{n}{n+1}+\frac{f(n)+1}{n+1}\right)}{\frac{n}{n+1}+\frac{f(n)+1}{n+1}}
=\frac{n+1}{f(n)+1}
=\frac{n+1}{f(n+1)+1}.
\end{align*}
Note that we use $f(n)=f(n+1)$ in the first and the last equality.
\end{proof}

We are ready to prove \cref{thm:complete_graph}. 
\begin{proof}[Proof of \cref{thm:complete_graph}]
Recall that $\E[U]=S(n)$.
Statement~\ref{item:complete_Omega(i)}
follows from  \cref{lem:Snbound}\ref{item:upper}.
Statement~\ref{item:complete_o(i)}
follows from \ref{item:tightlowerbound}
and \ref{item:tightupperbound} of \cref{lem:Snbound}.
\ref{item:complete_const}
follows from
\ref{item:tightlowerbound}
and \ref{item:tightupperbound2} of \cref{lem:Snbound}.

Now, we prove Statement~\ref{lab:complete_omega(i)}.
More precisely, we prove that,
for any $\epsilon>0$,
there is $n_0\in\Nat$
such that
for all $n\geq n_0$,
$S(n)\leq \epsilon$ holds.
From the assumption that $\func(i)=\omega(i)$,
for any large constant $C>0$,
we can take $i_0\in\Nat$
such that for all $i\geq i_0$, $f(i)>Ci$ holds.
Fix a constant $C>0$
and take $i_0$ in this way.
Since $1+x\leq \mathrm{e}^x$ and
$f(k)/k>C$ for all $k\geq i_0$,
we have
\begin{align*}
    S(n)&\leq \sum_{i=1}^{i_0}
    \exp\left(-\sum_{k=i_0}^n\frac{f(k)}{k}\right)
    +
    \sum_{i=i_0+1}^n
    \exp\left(-\sum_{k=i}^n\frac{f(k)}{k}\right) \\
    &\leq i_0\exp(-(n-i_0+1)C)+\sum_{i=i_0+1}^n \exp(-(n-i+1)C) \\
    &\leq i_0\exp(-(n-i_0+1)C)+\frac{\mathrm{e}^{-C}}{1-\mathrm{e}^{-C}}.
\end{align*}
Let $\epsilon>0$ be an arbitrary small constant.
Then, take $C>0$ such that
$\frac{\mathrm{e}^{-C}}{1-\mathrm{e}^{-C}}<\frac{\epsilon}{2}$
holds.
According to this constant $C$,
we can take $i_0$ such that $f(i)>Ci$ for all $i\geq i_0$ holds.
Now $C$ and $i_0$ are fixed.
Hence, for sufficiently large $n$,
we have $i_0\exp(-(n-i_0+1)C)\leq \frac{\epsilon}{2}$.
This implies $S(n)\leq \epsilon$
and we are done.
\end{proof}

\paragraph{Remark.}
 We remark on some monotonicity of $\E[U]$ with respect to $\func$. 
 Suppose functions $\func^*$ and $\func$ satisfy 
  $\func^*(i)\geq \func(i)$ for all $i$. 
 Let $U^*(n)$ and $U(n)$ respectively denote the numbers of unvisited vertices at the end of $n$-th round 
  for $R^*_{\rm c}=(\func^*,(G^{(i)})_{i=1}^{\infty},(P^{(i)})_{i=1}^{\infty})$ and 
  $R_{\rm c}=(\func,(G^{(i)})_{i=1}^{\infty},(P^{(i)})_{i=1}^{\infty})$. 
 Then, $\E[U^*(n)]\leq \E[U(n)]$ is clear. 
From this observation, 
 \cref{lem:Snbound} implies the following proposition, 
  which is a variant of \cref{thm:complete_graph} \ref{item:complete_Omega(i)}, \ref{item:complete_o(i)} and \ref{item:complete_const}.  
\begin{proposition}
\label{cor:simpleKn}
Let $C>0,\gamma\in [0,1]$ be arbitrary constants.
For $R_{\mathrm{c}}=(\func,(G^{(i)})_{i=1}^\infty,(P^{(i)})_{i=1}^\infty)$, the following holds:
\begin{enumerate}[label=(\arabic*)]
\item \label{lab:simpleKnupper} If $\func(i)\geq Ci^{1-\gamma}$ for all $i$, then $\E[U]\leq \frac{n^\gamma}{C}$.
\item \label{lab:simpleKnlower} If $\func(i)\leq Ci^{1-\gamma}$ for all $i$, then $\E[U]\geq \frac{n^{\gamma}}{C+n^{\gamma-1}}\left(1-\frac{1}{n}\right)^{Cn^{1-\gamma}}\geq \frac{n^\gamma}{C}-1$.
\end{enumerate}
\end{proposition} 
\section{Upper Bound Analysis}
\label{sec:general_graphs}
In this section we show \cref{thm:simplified_main,thm:tmix<<thit,thm:symmetry,thm:simplelazyRW}.
Consider a random walk on a growing graph $R=(\func, (G^{(n)})_{n=1}^{\infty}, (P^{(n)})_{n=1}^{\infty})$.
Recall that, at each round $i$, $(X_t^{(i)})_{t=0}^\infty$ denotes the random walk according to $P^{(i)}$ where $X^{(i)}_0=X^{(i-1)}_{\func(i-1)}$ holds (See \cref{tbl:ZtoX} for an example). 
Let $\pi^{(i)}$ denote the stationary distribution of $P^{(i)}$.
Let $\tau^{(i)}_v\defeq \min\{t\geq 0:X_t^{(i)}=v\}$, i.e., $\tau^{(i)}_v$ denotes the time taken for a random walk $(X_t^{(i)})_{t=0}^\infty$ to reach $v\in V^{(i)}$.
Note that $\thit(i)=\max_{u,v\in V}\E[\tau_v^{(i)}| X_0^{(i)}=u]$.
Suppose that the initial position is fixed, i.e., $X_0^{(1)}=v_1$.
For any round $k\leq n$, the probability that the walker does not visit the vertex $v_k$ until the end of the round $n$ is equal to $\Pr\left[\bigwedge_{i=k}^n\left\{\tau^{(i)}_{v_k}>\func(i)\right\}\right]$.
Hence we have
\begin{align}
\E[U]
&=\sum_{k=1}^n\Pr\left[\bigwedge_{i=k}^n\left\{\tau^{(i)}_{v_k}>\func(i)\right\}\right]
=\sum_{k=2}^n\Pr\left[\bigwedge_{i=k}^n\left\{\tau^{(i)}_{v_k}>\func(i)\right\}\right] \nonumber\\
&=\sum_{k=2}^n\sum_{v\in V^{(k-1)}}\Pr\left[X_0^{(k)}=v\right]\Pr\left[\bigwedge_{i=k}^n\left\{\tau^{(i)}_{v_k}>\func(i)\right\}\middle| X_0^{(k)}=v\right] \label{eq:upperEU_gen_prev}\\
&\leq \sum_{k=2}^n\max_{v\in V^{(k-1)}}\Pr\left[\bigwedge_{i=k}^n\left\{\tau^{(i)}_{v_k}>\func(i)\right\}\middle| X_0^{(k)}=v\right].
\label{eq:upperEU_gen}
\end{align}
The second equality follows from $\Pr[X_1^{(1)}\neq v_1]=0$.
The rest of this section is devoted to give upper bounds of \cref{eq:upperEU_gen} (or \cref{eq:upperEU_gen_prev}).
\subsection{Upper bound for large $\func$}
\label{sec:simple_upper}
We show \cref{thm:simplified_main} in this section.
To begin with, we show the following useful lemma.
\begin{lemma}
\label{lem:unvisit_hit}
For any $R=(\func, (G^{(i)})_{i=1}^{\infty}, (P^{(i)})_{i=1}^{\infty})$, we have
\begin{align*}
\E[U]&\leq \sum_{k=2}^n\prod_{i=k}^n\max_{v\in V^{(i)}}\Pr\left[\tau^{(i)}_{v_k}>\func(i)\middle|X_0^{(i)}=v\right].
\end{align*}
\end{lemma}
\begin{proof}
Consider a fixed vertex $v_k$ with $k>1$.
For a round $i\geq k$ and a vertex $u\in V^{(i)}$, let $\mathcal{E}^{(i)}_{u}=\mathcal{E}^{(i)}_{u}(v_k)$ denote the event
that the walker is in vertex $u$ at the end of the $i$-th round without
visiting vertex $v_k$ during the round.
Formally $\mathcal{E}^{(i)}_u(v_k)$ is defined as the event of 
$\{\tau^{(i)}_{v_k}>\func(i)\}\wedge \{X_{\func(i)}^{(i)}=u\}$.
Then for any $u_{k-1}\in V^{(k-1)}$, 
\begin{align}
\Pr\left[\bigwedge_{i=k}^n\left\{\tau^{(i)}_{v_k}>\func(i)\right\}\middle| X_0^{(k)}=u_{k-1}\right]
&=\sum_{u_k\in V^{(k)}} \cdots \sum_{u_n\in V^{(n)}}\Pr\left[\bigwedge_{i=k}^n\mathcal{E}^{(i)}_{u_i}\middle| X_0^{(k)}=u_{k-1}\right]. 
\label{eq:summation_vi}
\end{align}
To bound \cref{eq:summation_vi}, we first observe that, for any vertices, $u_{k-1}\in V^{(k-1)},u_k\in V^{(k)},\ldots,u_n\in V^{(n)}$,
\begin{align}
\Pr\left[\bigwedge_{i=k}^n\mathcal{E}^{(i)}_{u_i}\middle| X_0^{(k)}=u_{k-1}\right]
&=\frac{\Pr\left[X_0^{(k)}=u_{k-1}, \mathcal{E}^{(k)}_{u_k}\right]}{\Pr\left[X_0^{(k)}=u_{k-1}\right]}\prod_{\ell=k+1}^n\frac{\Pr\left[X_0^{(k)}=u_{k-1},\bigwedge_{i=k}^\ell\mathcal{E}^{(i)}_{u_i}\right]}{\Pr\left[X_0^{(k)}=u_{k-1},\bigwedge_{i=k}^{\ell-1}\mathcal{E}^{(i)}_{u_i}\right]}
\label{eq:product_canceling}
\end{align}
holds. Then, from the definition of the conditional probability, we have $\frac{\Pr\left[X_0^{(k)}=u_{k-1},\mathcal{E}^{(k)}_{u_k}\right]}{\Pr\left[X_0^{(k)}=u_{k-1}\right]}=\Pr[\mathcal{E}^{(k)}_{u_k}|X_0^{(k)}=u_{k-1}]$
and 
\begin{align}
\frac{\Pr\left[X_0^{(k)}=u_{k-1},\bigwedge_{i=k}^\ell\mathcal{E}^{(i)}_{u_i}\right]}{\Pr\left[X_0^{(k)}=u_{k-1},\bigwedge_{i=k}^{\ell-1}\mathcal{E}^{(i)}_{u_i}\right]}
&=\Pr\left[\mathcal{E}^{(\ell)}_{u_\ell}\middle|X_0^{(k)}=u_{k-1},\bigwedge_{i=k}^{\ell-1}\mathcal{E}^{(i)}_{u_i}\right] \nonumber \\
&=\Pr\left[\mathcal{E}^{(\ell)}_{u_\ell}\middle|X_{\func(\ell-1)}^{(\ell-1)}=u_{\ell-1}\right]
=\Pr\left[\mathcal{E}^{(\ell)}_{u_\ell}\middle|X_{0}^{(\ell)}=u_{\ell-1}\right].
\label{eq:Markov_property}
\end{align}
We use the Markov property in the second equality.
The last equality follows from our assumption of
$X_{\func(\ell-1)}^{(\ell-1)}=X^{(\ell)}_0$.
Hence combining \cref{eq:summation_vi,eq:product_canceling,eq:Markov_property}, 
we have
\begin{align}
&\Pr\left[\bigwedge_{i=k}^n\left\{\tau^{(i)}_{v_k}>\func(i)\right\}\middle| X_0^{(k)}=u_{k-1}\right] \nonumber \\
&=\sum_{u_k\in V^{(k)}} \cdots \sum_{u_n\in V^{(n)}} \prod_{\ell=k}^n\Pr\left[\tau^{(\ell)}_{v_k}>f(\ell), X_{\func(\ell)}^{(\ell)}=u_{\ell} \middle|X_{0}^{(\ell)}=u_{\ell-1}\right]
\label{eq:knotvisited_rounds_equal} \\
&=\sum_{\substack{u_k\in V^{(k)}}} \Pr\left[\mathcal{E}^{(k)}_{u_k}\middle|X_{0}^{(k)}=u_{k-1}\right]
\sum_{\substack{u_{k+1}\in V^{(k+1)}}} \Pr\left[\mathcal{E}^{(k+1)}_{u_{k+1}}\middle|X_{0}^{(k+1)}=u_{k}\right]
\cdots  \sum_{\substack{u_{n}\in  V^{(n)}}} \Pr\left[\mathcal{E}^{(n)}_{u_{n}}\middle|X_{0}^{(n)}=u_{n-1}\right] \nonumber \\
&\leq \prod_{\ell=k}^n \max_{u\in V^{(\ell)}}\sum_{u_\ell\in V^{(\ell)}}\Pr\left[\mathcal{E}^{(\ell)}_{u_{\ell}}\middle|X_{0}^{(\ell)}=u\right]
= \prod_{\ell=k}^n \max_{u\in V^{(\ell)}} \Pr\left[\tau_{v_k}^{(\ell)}>\func(\ell) \middle| X_0^{(\ell)}=u\right].
\label{eq:knotvisited_rounds_upper} \end{align}
We obtain the claim from \cref{eq:upperEU_gen,eq:knotvisited_rounds_upper}.
\end{proof}
\begin{proof}[Proof of \cref{thm:simplified_main}\ref{lab:general}]
From the Markov inequality, for any $k\leq i$ and $v\in V^{(i)}$, we have
\begin{align*}
\Pr\left[\tau_{v_k}^{(i)}>\func(i)\middle|X_0^{(i)}=v\right]
\leq \frac{\E\left[\tau_{v_k}^{(i)}\middle|X_0^{(i)}=v\right]}{\func(i)}
\leq \frac{\thit(i)}{\func(i)}.
\end{align*}
Hence from \cref{lem:unvisit_hit}, we obtain
\begin{align*}
\E[U]
&\leq \sum_{k=1}^n\prod_{i=k}^n\frac{\thit(i)}{\func(i)}
\leq \sum_{k=1}^nC^{-(n-k+1)}
=\sum_{k=1}^nC^{-k}
\leq \frac{1}{C-1}.
\end{align*}
\end{proof}
\begin{proof}[Proof of \cref{thm:simplified_main}\ref{lab:general_omega}]
For an arbitrary (small) $\epsilon>0$, let $C=C(\epsilon)=\frac{2}{\epsilon}+1$.
From assumption on \ref{lab:general_omega}, 
we can take some $i_0=i_0(\epsilon)$ such that $\func(i)\geq C\thit(i)$ for all $i\geq i_0$.
Let $K=\max_{i\in [i_0]}\frac{\thit(i)}{\func(i)}$.
From~\cref{lem:unvisit_hit}, 
\begin{align*}
\E[U]
&\leq \sum_{i=1}^{i_0}\left(\prod_{k=i}^{i_0}\frac{\thit(k)}{\func(k)}\right)\left(\prod_{k=i_0+1}^{n}\frac{\thit(k)}{\func(k)}\right)+\sum_{i=i_0+1}^{n}\prod_{k=i}^n\frac{\thit(k)}{\func(k)}\\
&\leq C^{-(n-i_0)}\sum_{i=1}^{i_0}K^{i-i_0+1}+\sum_{i=i_0+1}^nC^{-(n-i+1)}\\
&=C^{-(n-i_0)}\sum_{i=1}^{i_0}K^{i}+\sum_{i=1}^{n-i_0}C^{-i}\\
&\leq C^{-(n-i_0)}\frac{K(1-K^{i_0})}{1-K}+\frac{1}{C-1}.
\end{align*}
Then we can take some $n_0=n_0(\epsilon)$ satisfying $C^{-(n-i_0)}\frac{K(1-K^{i_0})}{1-K}\leq \epsilon/2$.
Hence for any $n\geq n_0$, $\E[U]\leq \epsilon$ and we obtain the claim.
\end{proof}
\subsection{Upper bound for random walks with small mixing times}
\label{sec:expander}
In this section we show the following generalized version of  \cref{thm:tmix<<thit}.
\begin{theorem}
\label{thm:tmix<<thit_gen}
Suppose that $P^{(i)}$ is reversible and lazy in $R=(\func, (G^{(i)})_{i=1}^{\infty}, (P^{(i)})_{i=1}^{\infty})$.
Let $\Den>0$ be an arbitrary positive number. 
If $\func(i)\geq \frac{\thit(i)}{\Den}+2\tmix(i) $ for all $i\in [n]$, then $\E[U]\leq 8\Den+32$.
\end{theorem}
\begin{proof}[Proof of \cref{thm:tmix<<thit}]
For all $i$, it is straight forward to see that
\begin{align*}
    \func(i)\geq \frac{C\thit(i)}{i^\gamma}+\frac{2C\thit(i)}{i^\gamma}\geq \frac{\thit(i)}{n^\gamma/C}+2\tmix(i)
\end{align*}
from assumptions. Taking $\Den=n^\gamma/C$ in \cref{thm:tmix<<thit_gen}, we obtain the claim.
\end{proof}
To show \cref{thm:tmix<<thit_gen}, we introduce following two lemmas.
The first one generalizes \cref{lem:Snbound}\ref{item:upper}.
The second one is a useful variant of \cref{lem:unvisit_hit}.
\begin{lemma}
\label{lem:Kn}
For $f,h:\Nat\to \Nat$ and $n\in \Nat$, let
\begin{align*}
S(n)\defeq \sum_{k=1}^n \prod_{i=k}^n \left(1-\frac{1}{h(i)}\right)^{f(i)}.
\end{align*}
Let $\Den>0$ be an arbitrary number.
If $f(i)\geq \frac{h(i)}{\Den}$ for all $i\in [n]$, then $S(n)\leq \Den$.
\end{lemma}
\begin{proof}
It is easy to check that
\begin{align*}
S(n)
&\leq \sum_{k=1}^n \prod_{i=k}^n \exp\left(-\frac{f(i)}{h(i)}\right)
= \sum_{k=1}^n \exp\left(-\sum_{i=k}^n\frac{f(i)}{h(i)}\right)
\leq \sum_{k=1}^n \exp\left(-\frac{n+k-1}{\Den}\right)\\
&= \sum_{k=1}^n \exp\left(-\frac{k}{\Den}\right)
\leq \frac{e^{-1/\Den}}{1-e^{-1/\Den}}
=\frac{1}{e^{1/\Den}-1}
\leq \Den.
\end{align*}
Note that we use $1+x\leq \mathrm{e}^x$ in the first and the last inequalities.
\end{proof}
\begin{lemma}
\label{lem:unvisit_mix_hit}
For any $R=(\func, (G^{(i)})_{i=1}^{\infty}, (P^{(i)})_{i=1}^{\infty})$ and any function $s:\mathbbm{N}\to \mathbbm{N}$ such that $s(i)<\func(i)$ holds for all $i$, we have 
\begin{align*}
\E[U]&\leq \sum_{k=2}^n\prod_{i=k}^n\max_{u\in V^{(i)}}\left(\sum_{v\in V^{(i)}}\left((P^{(i)})^{s(i)}\right)_{u,v}\Pr\left[\tau^{(i)}_{v_k}>\func(i)-s(i)\middle|X_0^{(i)}=v\right]\right).
\end{align*}
\end{lemma}
\begin{proof}
Fix $k\geq 2$ and $i$ satisfying $k\leq i\leq n$.
First, for any $u,v\in V^{(i)}$, from the definition of the conditional probability, we observe that
\begin{align*}
\lefteqn{\Pr\left[\tau^{(i)}_{v_k}>\func(i),X_{s(i)}^{(i)}=v\middle|X_0^{(i)}=u\right]}\\
&=\Pr\left[\tau^{(i)}_{v_k}>\func(i)\middle|X_{s(i)}^{(i)}=v,X_0^{(i)}=u,\tau^{(i)}_{v_k}>s(i)\right]
\Pr\left[X_{s(i)}^{(i)}=v,\tau^{(i)}_{v_k}>s(i)\middle|X_0^{(i)}=u\right]\\
&=\Pr\left[\tau^{(i)}_{v_k}>\func(i)-s(i)\middle|X_0^{(i)}=v\right]
\Pr\left[X_{s(i)}^{(i)}=v,\tau^{(i)}_{v_k}>s(i)\middle|X_0^{(i)}=u\right]
\end{align*}
holds. 
We use the Markov property in the third equality.
Since 
$$
\Pr\left[X_{s(i)}^{(i)}=v,\tau^{(i)}_{v_k}>s(i)\middle|X_0^{(i)}=u\right]\leq\Pr\left[X_{s(i)}^{(i)}=v\middle|X_0^{(i)}=u\right]= ((P^{(i)})^{s(i)})_{u,v},
$$ we have
\begin{align}
\Pr\left[\tau^{(i)}_{v_k}>\func(i)\middle|X_0^{(i)}=u\right]
&=\sum_{v\in V^{(i)}}\Pr\left[\tau^{(i)}_{v_k}>\func(i),X_{s(i)}^{(i)}=v\middle|X_0^{(i)}=u\right]\nonumber \\
&\leq \sum_{v\in V^{(i)}}\left((P^{(i)})^{s(i)}\right)_{u,v}\Pr\left[\tau^{(i)}_{v_k}>\func(i)-s(i)\middle|X_0^{(i)}=v\right]
\label{eq:splithit}
\end{align}
for any $u\in V^{(i)}$. 
Combining \cref{lem:unvisit_hit,eq:splithit}, we obtain the claim. 
\end{proof}
\begin{proof}[Proof of \cref{thm:tmix<<thit_gen}]
If $P^{(i)}$ is reversible, for any $i\in [n]$ and $u,v\in V^{(i)}$, some transition matrix $\hat{P}^{(i)}\in [0,1]^{V^{(i)}\times V^{(i)}}$ exists such that
\begin{align}
\left((P^{(i)})^{2\tmix(i)}\right)_{u,v}=\frac{1}{4}\pi^{(i)}(v)+\frac{3}{4}(\hat{P}^{(i)})_{u,v}
\label{eq:separation}
\end{align}
holds (See e.g., p.338 of \cite{LP17}).
Hence it holds for any $u\in V^{(i)}$ that
\begin{align}
\lefteqn{\sum_{v\in [i]}\left((P^{(i)})^{2\tmix(i)}\right)_{u,v}\Pr\left[\tau_{v_k}^{(i)}>\func(i)-2\tmix(i)\middle|X_0^{(i)}=v\right]}\nonumber\\
&= \frac{1}{4}\sum_{v\in [i]}\pi^{(i)}(v)\Pr\left[\tau_{v_k}^{(i)}>\func(i)-2\tmix(i)\middle|X_0^{(i)}=v\right]+\frac{3}{4}\sum_{v\in [i]}(\hat{P}^{(i)})_{u,v}\Pr\left[\tau_{v_k}^{(i)}>\func(i)-2\tmix(i)\middle|X_0^{(i)}=v\right]\nonumber\\
&\leq \frac{1}{4}\exp\left(-\frac{\func(i)-2\tmix(i)}{\thit(i)}\right)+\frac{3}{4}
\leq \frac{1}{4}\exp\left(-\frac{1}{\Den}\right)+\frac{3}{4}.
\label{eq:mixhit_split}
\end{align}
We use \cref{lem:multihit} in the first inequality.
Now, for a positive integer $L$, consider a random variable $X\sim Bin(L,1/4)$.
Here, $Bin(L,1/4)$ is the binomial distribution with parameters $L$ and $1/4$.
Then, it is straightforward to see that
\begin{align}
\left(\frac{1}{4}\exp\left(-\frac{1}{\Den}\right)+\frac{3}{4}\right)^L
&=\sum_{i=0}^L\left(\begin{array}{c}L\\i\end{array}\right)\left(\frac{1}{4}\exp\left(-\frac{1}{\Den}\right)\right)^i\left(\frac{3}{4}\right)^{L-i}
=\sum_{i=0}^L\exp\left(-\frac{i}{\Den}\right)\Pr\left[X=i\right]\nonumber\\
&\leq \sum_{i=0}^{\lfloor L/8\rfloor}\exp\left(-\frac{i}{\Den}\right)\Pr\left[X=i\right]+
\sum_{i=\lceil L/8\rceil}^{L}\exp\left(-\frac{i}{\Den}\right)\Pr\left[X=i\right] \nonumber \\
&\leq \Pr\left[X\leq \frac{L}{8}\right] +\exp\left(-\frac{L}{8\Den}\right) \leq \exp\left(-\frac{L}{32}\right)+\exp\left(-\frac{L}{8\Den}\right).
\label{eq:binomial}
\end{align}
The last inequality follows since 
\begin{align*}
\Pr\left[X\leq \frac{L}{8}\right]
=\Pr\left[X\leq \frac{\E[X]}{2}\right]
\leq \exp\left(-\frac{\E[X]}{8}\right)=\exp\left(-\frac{L}{32}\right)
\end{align*}
holds from the Chernoff inequality \cref{lem:Chernoff}.
Thus combining \cref{lem:unvisit_mix_hit,eq:mixhit_split,eq:binomial}, we obtain
\begin{align*}
\E[U]
&\leq \sum_{k=1}^n\left(\frac{1}{4}\exp\left(-\frac{1}{\Den}\right)+\frac{3}{4}\right)^{n-k+1}\leq \sum_{k=1}^n\left(\exp\left(-\frac{n-k+1}{32}\right)+\exp\left(-\frac{n-k+1}{8\Den}\right)\right)\\
&=\sum_{k=1}^n\exp\left(-\frac{k}{32}\right)+\sum_{k=1}^n\exp\left(-\frac{k}{8\Den}\right) \leq 32+8\Den.
\end{align*}
\end{proof}
\paragraph{Example: Degree restricted expander graph.}
For a graph $G=(V,E)$, let $d_{\textrm{ave}}(G)$ and $d_{\textrm{min}}(G)$ denote the average and the minimum degree of $G$,  respectively.
Suppose that $P$ is the transition matrix of the lazy  simple random walk on $G$ and let $\lambda_2(P)$ denote the second largest eigenvalue of $P$. 
We call a graph $G$ \emph{degree restricted expander graph} if both $\frac{d_{\textrm{ave}}(G)}{d_{\min}(G)}$ and
$\frac{1}{1-\lambda_2(P)}$ are upper bounded by some positive constant.
For any degree restricted expander graph, we have $\thit(P)=\Order(|V|)$ and $\tmix(P)=\Order(\log |V|)$ (See \cref{lem:thit_upper_lower} in \cref{app:tools} and Theorem 12.4 in \cite{LP17}).
Thus \cref{thm:tmix<<thit} implies the following.

\begin{corollary}
\label{cor:expander}
Suppose that $G^{(i)}$ is a degree restricted expander graph and $P^{(i)}$ is the transition matrix of the lazy simple random walk on $G^{(i)}$ in $R=(\func, (G^{(i)})_{i=1}^{\infty}), (P^{(i)})_{i=1}^{\infty})$.
Let $\gamma \in [0,1]$ and $C>0$ be arbitrary constants. 
Then two positive constants $K_1, K_2$ satisfying the following exist: 
If $\func(i)\geq C K_1 i^{1-\gamma} + K_2\log i$ for all $i\in [n]$, then $\E[U]\leq 8\frac{n^{\gamma}}{C}+32$.
\end{corollary}
\begin{proof}
Since there exist some positive constants $K_1, K_2$ satisfying $\thit(i)\leq K_1 i$ and $\tmix\leq K_2\log i$, we obtain the claim from \cref{thm:tmix<<thit_gen}.
\end{proof}

\subsection{Upper bounds for simple or symmetric random walks}
\label{sec:const}
This section is devoted to prove \cref{thm:upper_moderately}, which is a generalized version of \cref{thm:symmetry,thm:simplelazyRW}.

\begin{theorem}
\label{thm:upper_moderately}
Suppose that $P^{(i)}$ is reversible and lazy in $R=(\func, (G^{(i)})_{i=1}^{\infty}), (P^{(i)})_{i=1}^{\infty})$.
Let $r_i=\max_{v\in V^{(i-1)}}\frac{\pi^{(i-1)}(v)}{\pi^{(i)}(v)}$ for $1<i\leq n$.
Let $\Den$ be an arbitrary number.
If $\func(i)\geq  \left(\frac{1}{\Den}+\frac{i(r_i-1)+1}{2i}\right)\thit^{(i)}$ for all $i$, then $\E[U]\leq \Den\sqrt{\max_{1<i\leq n}i(r_i-1)+1}$.
\end{theorem}

\begin{proof}[Proof of \cref{thm:simplelazyRW}]
Let $d^{(i)}_v$ denote the degree of a vertex $v\in V^{(i)}$ at round $i$.
Then, for all $v\in V^{(i)}$, 
\begin{align*}
\frac{\pi^{(i-1)}(v)}{\pi^{(i)}(v)}=\frac{d^{(i-1)}_v}{2|E^{(i-1)}|}\frac{2|E^{(i)}|}{d^{(i)}_v}\leq \frac{|E^{(i)}|}{|E^{(i-1)}|}
\end{align*}
Note that $d^{(i-1)}_v\leq d^{(i)}_v$ holds from our assumption.
Combining the assumptions on $\func(i)$ and $E^{(i)}$, we have $\func(i)\geq \frac{\thit(i)}{i^\gamma/C}+\frac{L+1}{2i}\thit(i)\geq \frac{\thit(i)}{n^\gamma/C}+\frac{L+1}{2i}\thit(i)$.
 Thus we obtain the claim by taking $\Den=n^\gamma/C$ in \cref{thm:upper_moderately}.
\end{proof}
\begin{proof}[Proof of \cref{thm:symmetry}]
Since $P^{(i)}$ is symmetric, $r_i=\frac{i}{i-1}\leq 1+\frac{2}{i}$ for all $i>1$.
From the assumption of \cref{thm:symmetry}, $\func(i)\geq \frac{\thit(i)}{i^\gamma/C}+\frac{2\thit(i)}{i}\geq \frac{\thit(i)}{n^\gamma/C}+\frac{\thit(i)(2+1)}{2i}$ for all $1<i\leq n$. Thus we obtain the claim by taking $\Den=n^\gamma/C$  in \cref{thm:upper_moderately}.
\end{proof}
To show \cref{thm:upper_moderately}, we set the following notations.
For two vectors $f,g\in \mathbbm{R}$ and a probability vector $\pi\in (0,1]^V$, 
let $\inp<f,g>_\pi\defeq \sum_{v\in V}\pi(v)f(v)g(v)$. 
Then, the $\ell_2(\pi)$-norm of $f$ is defined by $\Nor<f>_{2,\pi}\defeq \sqrt{\inp<f,f>_\pi}=\sqrt{\sum_{v\in V}\pi(v)f(v)^2}$.
For two vectors $f,g\in \mathbbm{R}^V$ where $g(v)\neq 0$ holds for all $v\in V$, define $\frac{f}{g}\in \mathbbm{R}^V$ by $\frac{f}{g}(v)=\frac{f(v)}{g(v)}$.
Note that from these definitions, for any  probability vector $\xi\in [0,1]^V$,  $\Nor<\frac{\xi}{\pi}-\mathbbm{1}^{(|V|)}>_{2,\pi}^2=\Nor<\frac{\xi}{\pi}>_{2,\pi}^2-1$ holds.
Here, $\mathbbm{1}^{(n)}$ denotes the $n$-dimensional vector where all elements are equal to one. 
For a matrix $M\in \mathbbm{R}^{V\times V}$ let $\lambda_j(M)$ denote the $j$-th largest (in absolute value) eigenvalue of $M$. 

For any round $1<\ell\leq n$ and $0\leq t\leq \func(\ell)$, define a probability vector $\nu^{(\ell)}_t\in [0,1]^{V^{(\ell)}}$ where
\begin{align}
    \nu^{(\ell)}_t(v)&= \Pr[X_t^{(\ell)}=v] \hspace{1em} (\forall v\in V^{(\ell)}).
    \label{def:nu}
\end{align}
Furthermore, for any rounds $k, \ell$ satisfying $k-1\leq \ell \leq n-1$, define $\mu^{(\ell)}_{v_k}\in [0,1]^{V^{(\ell)}}$ by
\begin{align}
\mu_{v_k}^{(\ell)}(v)
= \Pr\left[\bigwedge_{i=\ell+1}^n\left\{\tau_{v_k}^{(i)}>\func(i)\right\}\middle| X_{\func(\ell)}^{(\ell)}=v\right] \hspace{1em} (\forall v\in V^{(\ell)})
\label{def:mu}
\end{align}
and $\mu^{(n)}_{v_k}\defeq \mathbbm{1}^{(n)}$.
%
Then, combining the Cauchy-Schwarz inequality, \cref{eq:upperEU_gen_prev,def:nu,def:mu}, we have
\begin{align}
    \E[U]
    &=\sum_{k=2}^n\sum_{v\in V^{(k-1)}}
    \nu^{(k-1)}_{\func(k-1)}(v)\mu^{(k-1)}_{v_k}(v)
    \leq \sum_{k=2}^{n}\sqrt{\sum_{v\in V^{(k-1)}}\frac{\nu^{(k-1)}_{\func(k-1)}(v)^2}{\pi^{(k-1)}(v)}\sum_{v\in V^{(k-1)}}\pi^{(k-1)}(v)\mu^{(k-1)}_{v_{k}}(v)^2} \nonumber \\
    &=\sum_{k=2}^{n}\Nor<\frac{\nu^{(k-1)}_{\func(k-1)}}{\pi^{(k-1)}}>_{2,\pi^{(k-1)}}\Nor<\mu^{(k-1)}_{v_{k}}>_{2,\pi^{(k-1)}}
    =\sum_{k=2}^{n}\sqrt{1+\Nor<\frac{\nu^{(k-1)}_{\func(k-1)}}{\pi^{(k-1)}}-\mathbbm{1}^{(k-1)}>_{2,\pi^{(k-1)}}^2}\Nor<\mu^{(k-1)}_{v_{k}}>_{2,\pi^{(k-1)}}.
    \label{eq:EU_bound_munuNorms}
\end{align}
In the rest of this section, 
we show the following bounds of $\Nor<\frac{\nu^{(k)}_{\func(k)}}{\pi^{(k)}}-\mathbbm{1}^{(k)}>_{2,\pi^{(k)}}$ and $\Nor<\mu^{(k-1)}_{v_{k}}>_{2,\pi^{(k-1)}}$, from which we immediately derive \cref{thm:upper_moderately}.

\begin{lemma}
\label{lem:L2_upper_gen}
Suppose that $P^{(i)}$ is reversible and lazy in $R=(\func, (G^{(i)})_{i=1}^{\infty}, (P^{(i)})_{i=1}^{\infty})$.
Let $r_i=\max_{v\in V^{(i-1)}}\frac{\pi^{(i-1)}(v)}{\pi^{(i)}(v)}$ for $1<i\leq n$.
If $\func(i)\geq \frac{i(r_i-1)+1}{2i(1-\lambda_2(P^{(i)}))}$, then $\Nor<\frac{\nu^{(k)}_{\func(k)}}{\pi^{(k)}}-\mathbbm{1}^{(k)}>_{2,\pi^{(k)}}^2< \max_{1<i\leq n}i(r_i-1)$ for all $1\leq k\leq n$.
\end{lemma}

\begin{lemma} 
\label{thm:upper_specific}
Suppose that $P^{(i)}$ is reversible and lazy in $R=(\func, (G^{(i)})_{i=1}^{\infty}, (P^{(i)})_{i=1}^{\infty})$.
For $1< i\leq n$, let $r_i\defeq \max_{v\in V^{(i-1)}}\frac{\pi^{(i-1)}(v)}{\pi^{(i)}(v)}$.
Let $\Den$ be an arbitrary positive number. 
If $\func(i)\geq \left(\frac{1}{\Den}+\frac{r_i-1}{2}\right)\thit(i)$ for all $1< i\leq n$, then
$
    \sum_{k=2}^n\Nor<\mu^{(k-1)}_{v_k}>_{2,\pi^{(k-1)}}
    \leq \Den
$
\end{lemma}


\begin{proof}[Proof of \cref{thm:upper_moderately}]
Suppose $\func(i)\geq \frac{\thit(i)}{\Den}+\frac{(i(r_i-1)+1)\thit(i)}{2i}$ for all $1<i\leq n$. Then, $\func(i)\geq \frac{i(r_i-1)+1}{2i(1-\lambda_2(P^{(i)})}$ from \cref{lem:thit_upper_lower}.
Furthermore, 
$\func(i)\geq \frac{\thit(i)}{\Den}+\frac{r_i-1}{2}\thit(i)$.
Thus applying \cref{thm:upper_specific,lem:L2_upper_gen} to \cref{eq:EU_bound_munuNorms}, 
\begin{align*}
    \E[U]\leq \sum_{k=2}^n\sqrt{\max_{1<i\leq n}i(r_i-1)+1}\Nor<\mu^{(k-1)}_{v_k}>_{2,\pi^{(k-1)}}\leq \Den\sqrt{\max_{1<i\leq n}i(r_i-1)+1}
\end{align*}
and we obtain the claim.
\end{proof}


%
%
\paragraph{Proof of \cref{lem:L2_upper_gen}}
%
First we show the following lemma.
This lemma gives a general upper bound of $\Nor<\frac{\nu^{(k)}_{\func(k)}}{\pi^{(k)}}-\mathbbm{1}^{(k)}>_{2,\pi^{(k)}}^2$ using $r_i$.
\begin{lemma}
\label{lem:increasing_l2discrepancy}
Suppose that $P^{(i)}$ is reversible and lazy in $R=(\func, (G^{(i)})_{i=1}^{\infty}, (P^{(i)})_{i=1}^{\infty})$.
Let $r_i=\max_{v\in V^{(i-1)}}\frac{\pi^{(i-1)}(v)}{\pi^{(i)}(v)}$  for $1<i\leq n$.
Then for any round $1\leq k\leq n$, 
\begin{align*}
\Nor<\frac{\nu^{(k)}_{\func(k)}}{\pi^{(k)}}-\mathbbm{1}^{(k)}>_{2,\pi^{(k)}}^2
\leq \sum_{i=2}^{k}\left(\prod_{j=i}^{k}r_j \lambda_2(P^{(j)})^{2\func(j)}\right)\left(1-\frac{1}{r_i}\right).
\end{align*}
\end{lemma}

\begin{proof}[Proof of \cref{lem:L2_upper_gen}]
First we observe that $\log \left(r_j(\frac{j+1}{j})\right)=\log (1+(r_j-1))+\log(1+\frac{1}{j})\leq (r_j-1)+\frac{1}{j}$. 
Hence it holds that
\begin{align*}
    \lambda_2(P^{(j)})^{2\func(j)}\leq \left(1-\left(1-\lambda_2(P^{(j)})\right)\right)^{\frac{\log \left(r_j(\frac{j+1}{j})\right)}{1-\lambda_2(P^{(j)})}}\leq \frac{1}{r_j}\cdot \frac{j}{j+1}.
\end{align*}
Applying \cref{lem:increasing_l2discrepancy}, we obtain
\begin{align*}
    \sum_{i=2}^{k}\left(\prod_{j=i}^{k}r_j\lambda_2(P^{(j)})^{2\func(j)}\right)\left(1-\frac{1}{r_i}\right)
    &\leq \sum_{i=2}^{k} \left(\prod_{j=i}^{k} \frac{j}{j+1}\right)\frac{r_i-1}{r_i}
    \leq \sum_{i=2}^k\frac{i}{k+1}(r_i-1)\\
    &\leq \max_{1<i\leq n}i(r_i-1)\frac{k-1}{k+1}<\max_{1<i\leq n}i(r_i-1).
\end{align*}
\end{proof}

\begin{proof}[Proof of \cref{lem:increasing_l2discrepancy}]
To obtain the claim, we show the following recurrence inequality:
\begin{align}
    \Nor<\frac{\nu^{(\ell)}_{\func(\ell)}}{\pi^{(\ell)}}-\mathbbm{1}^{(\ell)}>_{2,\pi^{(\ell)}}^2
    \leq r_\ell \lambda_2(P^{(\ell)})^{2\func(\ell)}\Nor<\frac{\nu^{(\ell-1)}_{\func(\ell-1)}}{\pi^{(\ell-1)}}-\mathbbm{1}^{(\ell-1)}>_{2,\pi^{(\ell-1)}}^2 + (r_\ell-1)\lambda_2(P^{(\ell)})^{2\func(\ell)}.
    \label{eq:rucurrence_ineq_nu}
\end{align}
Write $x_\ell=\Nor<\frac{\nu^{(\ell)}_{\func(\ell)}}{\pi^{(\ell)}}-\mathbbm{1}^{(\ell)}>_{2,\pi^{(\ell)}}^2$, $c_\ell=r_\ell \lambda_2(P^{(\ell)})^{2\func(\ell)}$ and $d_\ell=(r_\ell-1)\lambda_2(P^{(\ell)})^{2\func(\ell)}$ for notational convenience.
If \cref{eq:rucurrence_ineq_nu} holds for any $\ell>1$, 
applying \cref{eq:rucurrence_ineq_nu} repeatedly yields
\begin{align*}
x_k &\leq c_k x_{k-1} +d_{k} \leq c_kc_{k-1}x_{k-2} + c_kd_{k-1} + d_k 
\leq \cdots \leq 
\left(\prod_{i=2}^kc_i\right)x_1+\sum_{i=2}^k\left(\prod_{j=i+1}^kc_j\right)d_i.
\end{align*}
Since $x_1=\Nor<\frac{\nu^{(1)}_{\func(1)}}{\pi^{(1)}}-\mathbbm{1}^{(1)}>_{2,\pi^{(1)}}^2=0$ from definition, we obtain the claim.

Now we proceeds to show \cref{eq:rucurrence_ineq_nu}.
From the reversibility of $P^{(\ell)}$, it is easy to see that, for all $v\in V^{(\ell)}$, 
\begin{align}
\left(\frac{\nu^{(\ell)}_t}{\pi^{(\ell)}}\right)(v)
=\frac{\sum_{u\in V^{(\ell)}}\nu_0^{(\ell)}(u)\left((P^{(\ell)})^t\right)_{u,v}}{\pi^{(\ell)}(v)}
=\sum_{u\in V^{(\ell)}}\frac{\nu_0^{(\ell)}(u)\left((P^{(\ell)})^t\right)_{v,u}}{\pi^{(\ell)}(u)}
=\left((P^{(\ell)})^t\frac{\nu^{(\ell)}_0}{\pi^{(\ell)}}\right)(v).
\label{eq:nuPell}
\end{align}
From \cref{eq:nuPell,lem:Markov_variance}, it holds that 
\begin{align}
\Nor<\frac{\nu^{(\ell)}_{\func(\ell)}}{\pi^{(\ell)}}-\mathbbm{1}^{(\ell)}>_{2,\pi^{(\ell)}}^2
\leq \lambda_2(P^{(\ell)})^{2\func(\ell)}\Nor<\frac{\nu^{(\ell)}_{0}}{\pi^{(\ell)}}-\mathbbm{1}^{(\ell)}>_{2,\pi^{(\ell)}}^2
=\lambda_2(P^{(\ell)})^{2\func(\ell)}\left(\Nor<\frac{\nu^{(\ell)}_{0}}{\pi^{(\ell)}}>_{2,\pi^{(\ell)}}^2-1\right).
\label{eq:variance_nu_down}
\end{align}
Furthermore, for a vertex $v_\ell$ which appears at the round $\ell$, since  $\nu^{(\ell)}_0(v_\ell)=\Pr[X_0^{(\ell)}=v_\ell]=0$ holds, we have
\begin{align}
\Nor<\frac{\nu^{(\ell)}_{0}}{\pi^{(\ell)}}>_{2,\pi^{(\ell)}}^2
&=\sum_{v\in V^{(\ell-1)}}\pi^{(\ell)}(v)\frac{\nu^{(\ell)}_0(v)^2}{\pi^{(\ell)}(v)^2}
=\sum_{v\in V^{(\ell-1)}}\frac{\pi^{(\ell-1)}(v)}{\pi^{(\ell)}(v)}\pi^{(\ell-1)}(v)\frac{\nu^{(\ell-1)}_{\func(\ell-1)}(v)^2}{\pi^{(\ell-1)}(v)^2}\nonumber \\
&\leq r_\ell\sum_{v\in V^{(\ell-1)}}\pi^{(\ell-1)}(v)\frac{\nu^{(\ell-1)}_{\func(\ell-1)}(v)^2}{\pi^{(\ell-1)}(v)^2}
=r_\ell\Nor<\frac{\nu^{(\ell-1)}_{\func(\ell-1)}}{\pi^{(\ell-1)}}>_{2,\pi^{(\ell-1)}}^2.
\label{eq:down_round}
\end{align}
Combining \cref{eq:variance_nu_down,eq:down_round}, we obtain \cref{eq:rucurrence_ineq_nu}.
\end{proof}

\paragraph{Proof of \cref{thm:upper_specific}.}
The following lemma plays a key role in the proof of  \cref{thm:upper_specific}.
\begin{lemma}
\label{lem:munorm}
Suppose that $P^{(i)}$ is reversible and lazy in $R=(\func, (G^{(i)})_{i=1}^{\infty}, (P^{(i)})_{i=1}^{\infty})$.
For $1<i\leq n$, let $r_i=\max_{v\in V^{(i-1)}}\frac{\pi^{(i-1)}(v)}{\pi^{(i)}(v)}$.
Then, for any round $k$ satisfying $1<k\leq n$, 
\begin{align*}
\Nor<\mu^{(k-1)}_{v_k}>_{2,\pi^{(k-1)}}
\leq \prod_{i=k}^n\sqrt{r_i}\left(1-\frac{1}{\thit(i)}\right)^{\func(i)}.
\end{align*}
\end{lemma}
\begin{proof}[Proof of \cref{thm:upper_specific}]
Since $\log \sqrt{r_i}=\frac{1}{2}\log r_i = \frac{1}{2}\log (1+(r_i-1)) \leq \frac{r_i-1}{2}$, we have
\begin{align}
\sqrt{r_i}\left(1-\frac{1}{\thit(i)}\right)^{\func(i)}
\leq \left(1-\frac{1}{\thit(i)}\right)^{\func(i)-\thit(i)\log \sqrt{r_i}}
\leq \left(1-\frac{1}{\thit(i)}\right)^{\func(i)-\frac{r_i-1}{2}\thit(i)}.
\label{eq:thit_ri}
\end{align}
Thus combining \cref{lem:munorm,eq:thit_ri},
\begin{align*}
\sum_{k=2}^n\Nor<\mu^{(k-1)}_{v_k}>_{2,\pi^{(k-1)}}
\leq \sum_{k=2}^n\prod_{i=k}^n\left(1-\frac{1}{\thit(i)}\right)^{\func(i)-\frac{r_i-1}{2}\thit(i)}
\leq \Den.
\end{align*}
We invoke \cref{lem:Kn} in the last inequality.
\end{proof}
\begin{proof}[Proof of \cref{lem:munorm}]
For a transition matrix $P\in [0,1]^{V \times V}$ and a vertex $w \in V$, define $P_{\overline w}\in [0,1]^{V\times V}$ by
\begin{align*}
(P_{\overline w})_{u,v}=
\begin{cases}
P_{u,v} & (\textrm{if $u\neq w$ and $v\neq w$})\\
0 & (\textrm{otherwise})
\end{cases}.
\end{align*}
In other words, $(P_{\overline w})_{u,v}=P_{u,v}\mathbbm{1}_{u\neq w}\mathbbm{1}_{v\neq w}$ for $u,v\in V$.
Note that $P_{\overline w}$ is a substochastic matrix (see e.g., Section 3.6.5 of \cite{AF02}), i.e., $\sum_{v\in V}(P_{\overline w})_{u,v}\leq 1$ holds for any $u\in V$.
Observe for any $u,v\in V$ and $T>0$ that
\begin{align}
(P_{\overline w}^T)_{u,v}
&=\sum_{v_1\in V\setminus \{w\}}\cdots \sum_{v_{T-1}\in V\setminus \{w\}} \mathbbm{1}_{u\neq w}P_{u,v_1}P_{v_1,v_2}\cdots P_{v_{T-1},v} \mathbbm{1}_{v\neq w}\nonumber \\
&
=\Pr\left[\tau_{w}>T,X_T=v\middle|X_0=u\right].
\label{eq:kmatrix_prob}
\end{align}
Here, $(X_t)_{t=0}^\infty$ denotes a sequence of a random walk according to $P$ and $\tau_w$ denotes the hitting time to $w$. 
Note that $(P_{\setminus w}^T)_{u,v}=0$ if $u=w$ or $v=w$. 

Consider a fixed $k>1$.
Write $\mu^{(\ell)}=\mu^{(\ell)}_{v_k}$ and  $Q^{(\ell)}=(P^{(\ell)}_{\overline v_k})^{\func(\ell)}$ for notational convenience.
The key property for the proof is the following recurrence equation:
for all $k-1\leq \ell\leq n-1$ and   
$v\in V^{(\ell)}$, it holds that 
\begin{align}
\mu^{(\ell)}(v)=
\left(Q^{(\ell+1)}\mu^{(\ell+1)}\right)(v).
\label{eq:Matrix_mu}
\end{align}
This equation holds since for any $u_\ell\in V^{(\ell)}$, combining \cref{def:mu,eq:knotvisited_rounds_equal,eq:kmatrix_prob} yields 
\begin{align*}
\mu^{(\ell)}(u_\ell)
&=\sum_{u_{\ell+1}\in V^{(\ell+1)}}
\cdots \sum_{u_{n}\in V^{(n)}}\prod_{i=\ell+1}^n
\left((P^{(i)}_{v_k})^{\func(i)}\right)_{u_{i-1},u_{i}} \\
&=\sum_{u_{\ell+1}\in V^{(\ell+1)}}Q^{(\ell+1)}_{u_{\ell},u_{\ell+1}}\mu^{(\ell+1)}(u_{\ell+1}).
\end{align*}
Using \cref{eq:Matrix_mu,lem:opnorm}, we obtain
\begin{align}
\Nor<\mu^{(\ell)}>_{2,\pi^{(\ell)}}^2
&=\sum_{v\in V^{(\ell)}}\pi^{(\ell)}(v)\mu^{(\ell)}(v)^2
=\sum_{v\in V^{(\ell)}}\frac{\pi^{(\ell)}(v)}{\pi^{(\ell+1)}(v)}\pi^{(\ell+1)}(v)\left(Q^{(\ell+1)}\mu^{(\ell+1)}\right)(v)^2 \nonumber  \\
&\leq r_{\ell+1}\sum_{v\in V^{(\ell+1)}}\pi^{(\ell+1)}(v)\left(Q^{(\ell+1)}\mu^{(\ell+1)}\right)(v)^2
=r_{\ell+1} \Nor<Q^{(\ell+1)}\mu^{(\ell+1)}>_{2,\pi^{(\ell+1)}}^2 \nonumber \\
&\leq r_{\ell+1} \lambda_1(Q^{(\ell+1)})^2 \Nor<\mu^{(\ell+1)}>_{2,\pi^{(\ell+1)}}^2.
\label{eq:recuurenceQmu}
\end{align}
Hence applying \cref{eq:recuurenceQmu} repeatedly, it holds that
\begin{align}
\Nor<\mu^{(\ell)}>_{2,\pi^{(\ell)}}^2
\leq \prod_{i=\ell+1}^n r_{i} \lambda_1(Q^{(i)})^2.
\label{eq:mu_reccurence}
\end{align}
From the definition of $Q^{(i)}$ and $P^{(i)}_{\overline v_k}$, \cref{lem:upper_eigen} implies
\begin{align}
    \lambda_1(Q^{(i)})
    &=\lambda_1(P^{(i)}_{\overline v_k})^{\func(i)}
    \leq \left(1-\frac{1}{\thit(i)}\right)^{\func(i)}.
    \label{eq:upper_eigen_hit}
\end{align}
Thus we obtain the claim from \cref{eq:mu_reccurence,eq:upper_eigen_hit}.
\end{proof}

\paragraph{Example: Lollipop graph.}
Consider a growing lollipop graph:
We consider $G^{(i)}$ consisting of the complete graph $K_{\lceil i/2 \rceil}$ and the path graph $P_{\lfloor i/2\rfloor}$.
Formally, at each round $i\in [n]$, the set of odd vertices $V^{(i)}_o\defeq \{v_{2i+1}:1\leq i\leq \lceil i/2 \rceil\}$ forms the complete graph $K_{\lceil i/2 \rceil}$, the set of even vertices $V^{(i)}_e\defeq \{v_{2i}:1\leq i\leq \lfloor i/2 \rfloor\}$ forms a path graph, 
and these two components are connected by $\{v_1,v_2\}$.
Let $P^{(i)}$ be the transition matrix of the simple lazy random walk on $G^{(i)}$.
For such $P^{(i)}$, it is well known that $\thit(i)=\Order(i^3)$ (see e.g.~\cite{Feige95up}).
\begin{corollary}
\label{thm:lollipop}
Consider $R=(\func, (G^{(i)})_{i=1}^{\infty}, (P^{(i)})_{i=1}^{\infty})$ where $G^{(i)}$ is the lollipop graph defined above and $(P^{(i)})_{i\in [n]}$ is the transition matrix of the lazy simple random walk on $G^{(i)}$.
Let $\gamma \in [0,1]$ be an arbitrary constants.
If $\func(i)\geq C_1i^{3-\gamma}$ for all $i$, then $\E[U]\leq C_2n^{\gamma}$.
Here, $C_1, C_2$ are some positive constants.
\end{corollary}
\begin{proof}
From definition, $|E^{(2i)}|=1+\frac{i(i-1)}{2}+i-1=\frac{i(i+1)}{2}$ and $|E^{(2i+1)}|=1+\frac{(i+1)i}{2}+i-1=\frac{i(i+3)}{2}$.
Thus for any $i$, $\frac{|E^{(i)|}}{|E^{(i-1)|}}\leq 1+\frac{K_1}{i}$ for some constant $K_1$.
Furthermore, $\thit^{(i)}\leq K_2i^3$ holds for some constant $K_2$.
Applying \cref{thm:simplelazyRW}, we obtain the claim.
\end{proof}

\paragraph{Example: Metropolis walk.}
For a given $G=(V,E)$, the transition matrix of the lazy  Metropolis walk on $G$ is defined by 
\begin{align}
    (P)_{u,v}=
    \begin{cases}
    \frac{1}{2\max\{d_u,d_v\}} &(\textrm{if $\{u,v\}\in E$})\\
    1-\sum_{w:\{u,w\}\in E}(P)_{u,w}  &(\textrm{if $u=v$})\\
    0  &(\textrm{otherwise}).
    \end{cases} \label{def:Metro}
\end{align}
Due to the work of Nonaka, Ono, Sadakane and Yamashita~\cite{Nonaka10}, we have $\thit(P)=\Order(|V|^2)$ for any connected graphs.
Since $P$ is symmetric matrix, we can apply \cref{thm:symmetry}.
\begin{corollary}
\label{thm:Metro}
Suppose that $P^{(i)}$ is the lazy Metropolis walk on $G^{(i)}$ in $R=(\func, (G^{(i)})_{i=1}^{\infty}, (P^{(i)})_{i=1}^{\infty})$.
Let $\gamma \in [0,1]$ and $C>0$ be arbitrary constants.
If $\func(i)\geq \left(\frac{C}{i^\gamma}+\frac{2}{i}\right)\thit(i)$ for all $1<i\leq n$, then $\E[U]\leq \sqrt{3}\frac{n^\gamma}{C}$.
\end{corollary}


\newcommand{\pos}{\mathrm{pos}}
\section{A Lower Bound for a Growing Path} \label{sec:path}
This section is devoted to
prove \cref{thm:path_lowerbound}.
Let $L,R\in[n]$ be parameters satisfying $L< R$.
For a vertex $v\in V^{(n)}$, let $\mathcal{E}(v)$
be the event that $v\not\in \bigcup_{i=1}^n \bigcup_{t=0}^{\func(i)} \{X_t^{(i)}\}$.
In other words,
$\mathcal{E}(v)$ means that the walker
does not visit the vertex $v$ during the walk.
For two vertices $v_i,v_j\in V^{(n)}$,
we write $v_i\preceq v_j$ if $i\leq j$.
Note that, for any two vertices $u\preceq v$ and any round $k\in[n]$, it holds that
$\Pr[\mathcal{E}(v)|X^{(k)}_0\preceq u]\geq \Pr[\mathcal{E}(v)|X^{(k)}_0=u]$.
Then, we have
\begin{align}
    \E[U]&=\sum_{k=1}^n \Pr[\mathcal{E}(v_k)]
    \geq \sum_{k=R}^n\Pr[\mathcal{E}(v_k)] 
    \geq \sum_{k=R}^n\Pr\left[\mathcal{E}(v_k) \land X^{(k)}_0\preceq v_L\right] \nonumber\\
    &= \sum_{k=R}^n\Pr\left[\mathcal{E}(v_k) \middle| X^{(k)}_0\preceq v_L\right] \Pr[X^{(k)}_0\preceq v_L] \nonumber \\
    &\geq (n-R)\Pr\left[\mathcal{E}(v_R)\middle| X^{(R)}_0=v_L\right]\min_{R\leq k\leq n}\left\{\Pr\left[X^{(k)}_0\preceq v_L\right]\right\}. \label{eqn:EU_lowerbound_path}
\end{align}
We will determine the parameters $R$ and $L$ such that
$n-R=\Omega(n^\gamma)$,
$\Pr\left[\mathcal{E}(v_R)\middle| X^{(R)}_0=v_L\right]=\Omega(1/C)$ and $\Pr[X^{(k)}_0\leq L]=\Omega(1)$
for all $R\leq k\leq n$.
This yields the lower bound $\E[U]=\Omega(n^\gamma/C)$.
For fixed parameter $R$,
let $T\defeq \sum_{i=R}^n \func(i)$
denote the number of steps
of the walk during the last $n-R+1$ rounds.

\begin{lemma}\label{lem:pathlem}
Let $L,R\in\Nat$
be parameters satisfying $L<R$
and let $T\defeq \sum_{i=R}^n \func(i)$.
Then, the following holds.
\begin{enumerate}[label=(\roman*)]
\item\label{item:claim1_pathlem} $\Pr\left[\mathcal{E}(v_R)\middle| X^{(R)}_0=v_L\right]\geq 1-\frac{T}{4(R-L)^2}$, and
\item\label{item:claim2_pathlem} $\Pr[X^{(k)}_0\preceq v_L]\geq 1-\frac{L}{n}$ for all $k\in[n]$.
\end{enumerate}
\end{lemma}
\begin{proof}[Proof of \ref{item:claim1_pathlem}]
Let $(Z_t)_{t=1}^{\infty}$ be i.i.d.~random
variables sampled from the uniform distribution over $\{-1,+1\}$
and $S_c\defeq \sum_{j=0}^c Z_j$
denote the sum.
For a vertex $v_i\in V^{(n)}$, let $\pos(v_i)=i$ denote the position of $v_i$.
Then the complementary event $\overline{\mathcal{E}(v_R)}$ conditioned on $X^{(R)}_0=v_L$
is equivalent to the event that
$\max_{R\leq i\leq n,0\leq j\leq \func(i)}\{\pos(X^{(i)}_j)-\pos(X^{(R)}_0)\}\geq R-L$.
Moreover, the random variable
$\max_{R\leq i\leq n, 0\leq j\leq \func(i)}|\pos(X^{(i)}_j)-\pos(X^{(R)}_0)|$
is dominated\footnote{For two random variables $X$ and $Y$, we say $X$ \emph{dominates} $Y$ if, for any $r\in\mathbb{R}$, $\Pr[X\geq r] \geq \Pr[Y\geq r]$ holds.} by $\max_{1\leq c\leq T}|S_c|$ (recall $T=\sum_{i=R}^n \func(i)$).
This is because the distribution of $\pos(X^{(i)}_j)-\pos(X^{(i)}_{j-1})$ conditioned on $\pos(X^{(i)}_j)-\pos(X^{(i)}_{j-1})\neq 0$ is uniform on $\{-1,+1\}$.
Thus we obtain
\begin{align*}
    \Pr\left[\overline{\mathcal{E}(v_R)}
    \middle|
    X^{(R)}_0=v_L\right] 
    &\leq
    \Pr\left[\max_{R\leq i\leq n,0\leq j\leq \func(i)}|\pos(X^{(i)}_j)-\pos(X^{(R)}_0)|\geq R-L \middle| X^{(R)}_0=L\right] \\
    &\leq
    \Pr[\max_{1\leq c\leq T}|S_c|\geq R-L] \\
    &\leq \frac{\Var[S_T]}{(R-L)^2}
    = \frac{T}{4(R-L)^2}.
\end{align*}
In the last inequality,
we used the Kolmogorov inequality (\cref{lem:kolmogolov_ineq}).
\end{proof}
\begin{proof}[Proof of \ref{item:claim2_pathlem}]
It suffices to show that
\begin{align}
    \Pr[X^{(k)}_0=v_i] \geq \Pr[X^{(k)}_0=v_{i+1}] \label{eqn:monotone}
\end{align}
holds for any $1\leq i\leq k-1$.
To see this,
assuming \cref{eqn:monotone},
we obtain
\begin{align*}
    \frac{\Pr[X^{(k)}_0\preceq v_L]}{L}
    \geq
    \Pr[X^{(k)}_0=v_L]
    \geq
    \frac{1-\Pr[X^{(k)}_0\leq L]}{n-L},
\end{align*}
which implies the claim \ref{item:claim2_pathlem}.
Here, in the second inequality,
note that $\Pr[X^{(k)}_0=v_L] \geq \Pr[X^{(k)}_0=v_j]$ for all $j>L$
and thus,
the average $\frac{1}{n-L}\sum_{j>L}\Pr[X^{(k)}_0=v_j]$
is at most $\Pr[X^{(k)}_0=v_L]$.

Now we prove the inequality \cref{eqn:monotone}.
Let $x^{(i)}_j\in[0,1]^{V_i}$
denote the distribution of $X^{(i)}_j$.
To simplify notations, for a vector $y\in[0,1]^{V^{(i)}}$,
we write $y[u]$ for the $u$-th element of $y$.
We call the distribution $y\in[0,1]^{V^{(i)}}$ \emph{monotone}
if $y[v_k] \geq y[v_{k+1}]$ holds
for any $1\leq k\leq i-1$.
Our aim here is to prove that $x^{(k)}_0$ is monotone,
which is equivalent to \cref{eqn:monotone}.

Indeed, we will prove a stronger statement: $x^{(i)}_j$ is monotone for any $i$ and $j$.
We prove this statement inductively.
First, the vector $x^{(1)}_j=(1)$ is obviously monotone.
Secondly, if $x^{(i)}_{\func(i)}$ is monotone,
so does $x^{(i+1)}_0$.
To see this, note that $x^{(i+1)}_0$ is obtained by
concatenating $x^{(i)}_{\func(i)}$ with $0$.
More precisely, $x^{(i+1)}_0\in[0,1]^{i+1}$ satisfies
\begin{align*}
    x^{(i+1)}_0[j] = \begin{cases}
    x^{(i)}_{\func(i)}[j] & \text{if $1\leq j\leq i$},\\
    0 & \text{if $j=i+1$}.
    \end{cases}
\end{align*}
Finally, we check that $x^{(i)}_{j+1}$ is monotone if $x^{(i)}_j$ is monotone.
From \cref{eqn:transition_matrix_path}, we have
\begin{align*}
    x^{(i)}_{j+1}[v_k] = \begin{cases}
    px^{(i)}_j[v_1]+(1-p)x^{(i)}_j[v_2] & \text{if $k=1$}, \\
    qx^{(i)}_j[v_{k-1}]+(1-2q)x^{(i)}_j[v_k] + qx^{(i)}_j[v_{k+1}] & \text{if $1<k<i$},\\
    (1-p)x^{(i)}_j[v_{i-1}]+px^{(i)}_j[v_i] & \text{if $k=i$}.
    \end{cases}
\end{align*}
By the induction assumption, $x^{(i)}_j$ is monotone.
Now we check that $x^{(i)}_j$ is monotone.
For $k=1$, since $p\geq q$, we have
\begin{align*}
    x^{(i)}_{j+1}[v_1] - x^{(i)}_{j+1}[v_2]
    &= (p-q)(x^{(i)}_j[v_1]-x^{(i)}_j[v_2])+q(x^{(i)}_j[v_2]-x^{(i)}_j[v_3]) \geq 0.
\end{align*}
For $1<k<i-1$, since $q\leq \frac{1}{2}$, we have
\begin{align*}
    x^{(i)}_{j+1}[v_i]-x^{(i)}_{j+1}[v_{i+1}]
    &= qx^{(i)}_j[v_{k-1}]+(1-3q)x^{(i)}_j[v_k] - (1-3q)x^{(i)}_j[v_{k+1}]-qx^{(i)}_j[v_{k+2}] \\
    &\geq (1-2q)(x^{(i)}_j[v_k] - x^{(i)}_j[v_{k+1}]) \geq 0.
\end{align*}
Finally, for $k=i$, since $p\geq q$, we have
\begin{align*}
    x^{(i)}_{j+1}[v_{i-1}]-x^{(i)}_{j+1}[v_i]
    = q(x^{(i)}_j[v_{i-2}]-x^{(i)}_j[v_{i-1}])+(p-q)(x^{(i)}_j[v_{i-1}]-x^{(i)}_j[v_i]) \geq 0.
\end{align*}
Therefore $x^{(i)}_{j+1}$ is monotone.
\end{proof}

Now we are ready to prove
\cref{cor:path}.
Recall $\func(i)\leq Ci^{2-\gamma}$.
Fix a small positive constant $\epsilon$
such that $\epsilon< \min\{1/C,0.1\}$.
Set $R\defeq n-\epsilon n^{\gamma}$
and $L\defeq R-0.6n \in [0.3n,0.4n]$.
Then we have
$T\leq (n-R)\func(n) \leq C\epsilon n^2 \leq n^2$
and thus
$1-\frac{T}{4(R-L)^2} \geq 1-\frac{1}{4\times  0.36}>0.3$
and $1-\frac{L}{n}\geq 0.6$.
Then, from \cref{eqn:EU_lowerbound_path,lem:pathlem},
we have
\begin{align*}
    \E[U] \geq
    \epsilon n^\gamma \cdot 0.3 \cdot 0.6 = \Omega\left(\frac{n^\gamma}{C}\right),
\end{align*}
which completes the proof of \cref{thm:path_lowerbound} (here, we take $\epsilon>0$ such that $\epsilon=\Omega(1/C)$).
\hspace{\fill}\qedsymbol

\section{Concluding Remarks}
\label{sec:concluding}
 This paper has investigated
   the expected numbers of vertices remaining unvisited by random
walks on growing graphs parametrized by $\func$.
 We have presented some upper bounds of $\E[U]$ with respect to $\func$,
   where we revealed that
    $\E[U] = \Order(1)$ if $\func(i) \geq C \thit(i)$ for $C>1$ in
general (\cref{thm:simplified_main}), and that
    $\E[U] = \Order(1)$ if $\func(i) = \Omega(\thit(i))$
      on some natural assumptions
(\cref{thm:tmix<<thit,thm:simplelazyRW,thm:symmetry}).
 We have also presented lower bounds of  $\E[U]$
   for random walks on growing complete graphs and on growing path graphs,
    which imply the upper bounds by \cref{thm:tmix<<thit,thm:simplelazyRW,thm:symmetry} 
     are tight in those cases.
 A general lower bound of $\E[U]$ is a challenge:
  a natural question remains unsettled
   whether $\E[U] = \Order(1)$ requires $\func(i) = \Omega(\thit(i))$.
 A concentration result should be another future work~\cite{CF03}. 

 In this paper,
   we have been concerned with a simple model of graphs with the increasing number of vertices, 
   to develop a new technique for analyses of random walks on dynamic graphs.
 Clearly,
   it  is an interesting and important future work
    to analyze algorithms on dynamic graphs
    whose vertex set and edge set are both dynamic.
    
\section*{Acknowledgement}
This work is partly supported by JSPS KAKENHI Grant Numbers JP17K19982, JP19J12876, and JP19K20214.

\bibliographystyle{abbrv}
\bibliography{ref}

\appendix

\section{Note on the initial round}
\label{sec:initial_round}
For a $n_0>0$, we consider the case where $n_0$ vertices exist at the first round.
\begin{theorem}
Let $G^{(i)}=K_{n_0+i}$, i.e., the complete graph with $n_0+i$ vertices, and let $(P^{(i)})_{u,v}=1/(n_0+i)$ for all $u,v \in V^{(i)}$ in $R=(\func, (G^{(i)})_{i=1}^\infty, (P^{(i)})_{i=1}^{\infty})$. 
Let $\Den$ be an arbitrary positive number.
If $\func(i)\geq 2i/\Den$ for all $i$, then $\E[U(n)]\leq 2n_0+\Den$. 
\end{theorem}
\begin{proof}
If $n\leq n_0$, $|V^{(n)}|=n_0+n\leq 2n_0$ and we are done.
Suppose that $n>n_0$.
Then it is straightforward to see that
\begin{align*}
    \E[U(n)]&=n_0\prod_{i=1}^n\left(1-\frac{1}{n_0+i}\right)^{\func(i)}+
    \sum_{k=1}^{n}\prod_{i=k}^n\left(1-\frac{1}{n_0+i}\right)^{\func(i)}\\
    &\leq n_0 + n_0 + \sum_{k=n_0+1}^{n}\prod_{i=k}^n\left(1-\frac{1}{n_0+i}\right)^{\func(i)}\\
    &\leq 2n_0+ \sum_{k=n_0+1}^{n}\prod_{i=k}^n\left(1-\frac{1}{2i}\right)^{\func(i)}\\
    &\leq 2n_0 + \Den.
\end{align*}
Note that we use \cref{lem:Kn} in the last inequality.
\end{proof}

\section{Tools}\label{app:tools}
\begin{lemma}[The Kolmogorov inequality; Theorem 2.5.5 of \cite{Dur19}] \label{lem:kolmogolov_ineq}
Let $Z_1,\ldots,Z_n$ be i.i.d.~random
variables such that $\E[Z_i]=0$
and $\Var[Z_i]<\infty$.
Let $S_i=\sum_{j=1}^i Z_i$.
Then,
\begin{align*}
    \Pr[\max_{1\leq j\leq n} |S_j|\geq M ] \leq \frac{\Var[S_n]}{M^2}.
\end{align*}
\end{lemma}

\begin{lemma}[The Chernoff inequality (see e.g. Theorem 1.10.5 of \cite{Doerr18})]
\label{lem:Chernoff}
Let $X_1, X_2, \ldots, X_n$ be independent random variables taking values in $[0,1]$.
Let $X=\sum_{i=1}^nX_i$.
Let $\delta \in [0,1]$.
Then
\begin{align*}
\Pr\left[X\leq (1-\delta)\E[X]\right]\leq \exp\left(-\frac{\delta^2\E[X]}{2}\right).
\end{align*}
\end{lemma}

\begin{lemma}[See e.g.~Sections 2.4.3 of \cite{AF02}] \label{lem:static_unvisited}
Consider a random walk on a (static) graph $G=(V,E)$.
Then for any $c>0$ and any $v,u\in V$, 
$
    \Pr[\tau_v>c\mathrm{e}\thit|X_0=u]\leq \mathrm{e}^{-c}
$.
\end{lemma}
To see this, divide $c\mathrm{e}\thit$-steps
random walk into $c$ independent random walks
each of length $\mathrm{e}\thit$.
Then, in each walk, the walker
does not visit a specific vertex
with probability at most $1/\mathrm{e}$ from the Markov inequality.

Using \cref{lem:static_unvisited}, for any $t\geq \mathrm{e}\thit\log n$, it is easy to see that
$
    \E[\mathcal{U}_t]
    =\sum_{v\in V}\Pr[\tau_v>t|X_0=u]
    \leq  n \mathrm{e}^{-\log n}=1.
$
This implies that, for any RWoGG with $\func(i)\geq \mathrm{e}\thit(i)$, the number of unvisited vertices is at most $1$ in expectation at the end of every round.


\begin{lemma}[Theorem 4.1 of \cite{OP19}]
\label{lem:MeetingTimeLemma}
Let $P\in [0,1]^{V\times V}$ be an irreducible, reversible and lazy transition matrix over $V$, and let $\pi\in (0,1]^V$ denote its stationary distribution.
Let $(X_t)_{t=0}^\infty$ denote the Markov chain according to $P$.
Let $\tau_v(P)=\min\{t\geq 0:X_t=v\}$ and let $\thit(P)=\max_{u,v\in V}\E_u[\tau_v(P)]$.
Then for any $t\geq 0$ and any choice of $h_0, h_1, \ldots, h_t$, 
\begin{align*}
    \Pr_{\pi}\left[\forall 0\leq s\leq t: X_s\neq h_s\right]\leq \left(1-\frac{1}{\thit(P)}\right)^t.
\end{align*}
\end{lemma}
Taking $h_i=v\in V$ for all $0\leq i\leq t$ in \cref{lem:MeetingTimeLemma}, we immediately obtain the following.
\begin{corollary}
\label{lem:multihit}
Let $P\in [0,1]^{V\times V}$ be an irreducible, reversible and lazy transition matrix over $V$, and let $\pi\in (0,1]^V$ denote its stationary distribution.
Let $(X_t)_{t=0}^\infty$ denote the Markov chain according to $P$.
Let $\tau_v(P)=\min\{t\geq 0:X_t=v\}$ and let $\thit(P)=\max_{u,v\in V}\E_u[\tau_v(P)]$.
Then for any $v\in V$ and $t>0$, 
\begin{align*}
\Pr_{\pi}\left[\tau_v(P)>t\right]\leq \left(1-\frac{1}{\thit(P)}\right)^t
\leq \exp\left(-\frac{t}{\thit(P)}\right).
\end{align*}
\end{corollary}


\begin{lemma}[See Section 3.6.5 of \cite{AF02} or Theorem 4.1 of \cite{OP19}]
\label{lem:upper_eigen}
Let $P\in [0,1]^{V\times V}$ be an irreducible and reversible transition matrix over $V$, and let $\pi\in (0,1]^V$ denote its stationary distribution. 
For a subset $S\subseteq V$, define $P_{\overline S}\in [0,1]^{V\times V}$ by $(P_{\overline S})_{u,v}=P_{u,v}$ for any $u,v\in V\setminus S$ and $(P_{\overline S})_{u,v}=0$ for any $u\in S$ or $v\in S$.
Let $\lambda(M)$ denote the largest eigenvalue of a matrix $M$.
Then for any $S\notin \{\emptyset,V\}$,
\begin{align*}
\lambda(P_{\overline S})
&\leq 1-\frac{1}{\thit(P)}.
\end{align*}
Furthermore, for any $S\notin \{\emptyset,V\}$ and any $f\in \mathbbm{R}^V$, 
\begin{align*}
\Inp<f, P_{\overline S} f>_\pi \leq \lambda(P_{\overline S}) \Inp<f,f>_\pi.
\end{align*}
\end{lemma}
Since $\Nor<P_{\overline S}f>_{2,\pi}^2=\Inp<P_{\overline S} f, P_{\overline S} f>_\pi=\Inp< f, P_{\overline S}^2 f>_\pi$, 
we have the following corollary.
\begin{corollary}
\label{lem:opnorm}
Let $P\in [0,1]^{V\times V}$ be an irreducible, reversible and lazy transition matrix over $V$, and let $\pi\in (0,1]^V$ denote its stationary distribution. 
Suppose that $P_{\overline S}$ is a matrix defined in \cref{lem:upper_eigen}.
Then for any $S\notin \{\emptyset,V\}$ and any $f\in \mathbbm{R}^V$, 
\begin{align*}
\Nor<P_{\overline S}f>_{2,\pi}^2\leq \lambda_1(P_{\overline S})^2\Nor<f>_{2,\pi}^2
\leq \left(1-\frac{1}{\thit(P)}\right)^2 \Nor<f>_{2,\pi}^2
\end{align*}
Here, $\lambda_1(M)$ denotes the largest eigenvalue in absolute value of a matrix $M$.
\end{corollary}

\begin{lemma}[See e.g.~(12.8) of \cite{LP17}]
\label{lem:Markov_variance}
Let $P\in [0,1]^{V\times V}$ be a reversible transition matrix with respect to $\pi\in (0,1]^V$.
Then for any probability vector $f\in [0,1]^V$,
$\Nor<\frac{f}{\pi}-\mathbbm{1}>_{2,\pi}^2=\Nor<\frac{f}{\pi}>_{2,\pi}^2-1$ 
and
\begin{align*}
\Nor<P\frac{f}{\pi}-\mathbbm{1}>_{2,\pi}^2\leq \lambda_2(P)^2\Nor<\frac{f}{\pi}-1>_{2,\pi}^2
\end{align*}
holds where $\lambda_2(P)$ is the second largest eigenvalue (in absolute value) of $P$.
\end{lemma}



\begin{lemma}[Lemmas 4.24 and 4.25 of \cite{AF02}]
\label{lem:thit_upper_lower}
Let $P$ be reversible transition matrix and let $\pi$ be its stationary distribution. Then
\begin{align*}
    \frac{1}{1-\lambda_2(P)}\leq \thit(P) \leq \frac{2}{\pi_{\min}(1-\lambda_2(P))}.
\end{align*}
\end{lemma}

\end{document}